\documentclass[twoside]{article}
\usepackage[preprint]{aistats2027}
\usepackage[T1]{fontenc}
\usepackage[utf8]{inputenc}
\usepackage{lmodern}
\usepackage{amsmath,amssymb,amsthm,mathtools,graphicx,booktabs,microtype}
\usepackage[round]{natbib}
\usepackage{mathrsfs}
\usepackage[hidelinks]{hyperref}
\newtheorem{theorem}{Theorem}

\newtheorem{proposition}[theorem]{Proposition}
\newtheorem{corollary}[theorem]{Corollary}
\theoremstyle{definition}

\theoremstyle{remark}

\hypersetup{pdftitle={Sharp Initialization Guarantees for the Heavy-Ball Method},pdfauthor={}}
\begin{document}
\twocolumn[
\aistatstitle{Sharp Initialization Guarantees for the Heavy-Ball Method}
\aistatsauthor{Kang Liu$^1$, Quanen Wu$^2$, Zhan Sun$^3$, Mengxiao Chen$^1$, 
Suyan Li$^4$, 
Bohao Qu$^5$}
\aistatsaddress{
  $^1$Xi'an Jiaotong University (xjtu.edu.cn)\\
  $^2$Hong Kong Polytechnic University (polyu.edu.hk)\\
  $^3$Huazhong University of Science and Technology (hust.edu.cn)\\
  $^4$National University of Singapore (nus.edu)\\
  $^5$A*STAR (a-star.edu.sg)
}
]
\begin{abstract}
The heavy-ball method can accelerate convergence near the minimizer, but it may also enter persistent oscillations on general strongly convex objectives. How should it be initialized to reliably realize this local advantage? We study the initialization neighborhood in which convergence can be guaranteed using only curvature and regularity bounds. This neighborhood shrinks as momentum approaches one; for Lipschitz-continuous Hessians, its radius is of the same order as damping. Matching upper and lower bounds show that this initialization requirement cannot be relaxed uniformly over the function class. We then derive a startup rule based on the current gradient: a finite gradient-descent warm-up enters the guaranteed region before switching to heavy-ball. With parameters optimal for strongly convex quadratics, the subsequent iterates satisfy an accelerated local asymptotic rate bound. Code is available at \url{https://anonymous.4open.science/r/HeavyBall-811D/}.
\end{abstract}

\section{Introduction}
\label{sec:introduction}

Momentum is a classical mechanism for accelerating first-order optimization. The heavy-ball method incorporates momentum by combining the current gradient with the previous displacement:
\begin{equation}
    v_{k+1}=\beta v_k-\gamma\nabla f(x_k),
    \quad
    x_{k+1}=x_k+v_{k+1},
    \label{eq:hb}
\end{equation}
where $\gamma>0$ is the step size, $\beta\in[0,1)$ is the momentum parameter, and $v_k:=x_k-x_{k-1}$. For a $\mu$-strongly convex objective with an $L$-Lipschitz gradient, let $\kappa:=L/\mu$ denote its condition number. On quadratic objectives, \citet{polyak1964} established that suitable parameters yield the asymptotic convergence factor $(\sqrt{\kappa}-1)/(\sqrt{\kappa}+1)$. Beyond quadratics, however, stability of all quadratic models with curvature in $[\mu,L]$ does not ensure convergence from arbitrary initial states: \citet{lessard2016} exhibited an attracting periodic orbit on a one-dimensional strongly convex objective with a Lipschitz gradient.


Periodic obstructions also limit global rates. \citet{goujaud2025} proved that optimizing the fixed heavy-ball parameters does not improve the worst-case condition-number order of gradient descent, while \citet{fotopoulos2026} established the corresponding obstruction in one dimension through projected cycles. Locally, \citet{polyak1964} established convergence and asymptotic rate guarantees near a strongly convex minimizer; \citet{kassing2026} extended local acceleration to sufficiently smooth objectives satisfying the Polyak--{\L}ojasiewicz inequality. Further related work is discussed in Appendix~\ref{sec:related_work}. These results make initialization central to retaining the local benefits of momentum. The natural questions are \textit{how large a neighborhood can be guaranteed to converge uniformly, and how can one certify entry into that region?} 

We study strongly convex objectives whose Hessian eigenvalues lie in $[\mu,L]$ and whose Hessians are $\nu$-H\"older continuous with constant at most $M_\nu>0$, where $0<\nu\le1$. We define the \emph{uniform attraction radius} $r_{\mathrm{univ}}$ to quantify the initialization error that can be tolerated using only these curvature and regularity bounds: for every objective in the class, every initial position--velocity state at distance less than $r_{\mathrm{univ}}$ from equilibrium $(x_\star,0)$ converges to equilibrium; see \eqref{eq:uniform_radius} for the formal definition. We determine how this radius shrinks as damping $d:=1-\beta$ vanishes. A corrected energy controls the effect of curvature variation within the convergence region, while a scalar cycle interpolation shows how such variation can sustain periodic motion. The cycle amplitude determines the spatial obstruction, and its period determines which parameters admit periodic failure.

Our contributions are summarized as:
\begin{itemize}
    \item \textbf{Sharp attraction scale.}
    For fixed $\mu,L,\gamma,\nu$ with $\gamma L<4$, we prove
    $r_{\mathrm{univ}}
    =\Theta_{\mu,L,\gamma,\nu}
    (((1-\beta)/M_\nu)^{1/\nu})$ as $\beta\uparrow1$,
    with dimension-independent constants, also on the
    $C^\infty$ subclass.

    \item \textbf{Convergence and startup guarantees.}
    An explicit energy region guarantees linear convergence
    and yields a gradient-based startup test.
    Finite GD warm-up ensures entry; quadratic-optimal HB
    parameters then give an accelerated local asymptotic rate bound.

    \item \textbf{An elementary non-acceleration proof.}
    The same scalar cycle construction yields an elementary
    proof of the known global non-acceleration law.
\end{itemize}

\section{Problem Statement}
\label{sec:problem}

\subsection{Objectives and initialization}

We consider the optimization problem
\[
    \min_{x\in\mathbb R^n} f(x).
\]
Fix $0<\mu<L$, $M_\nu>0$, and $0<\nu\le1$. Let
$\mathcal F_{\mu,L,M_\nu}^{n,\nu}$ denote the class of twice continuously
differentiable functions $f:\mathbb R^n\to\mathbb R$ satisfying
\(
    \mu I\preceq \nabla^2 f(x)\preceq LI,
    x\in\mathbb R^n,
\)
and
\[
    \|\nabla^2 f(x)-\nabla^2 f(y)\|_{\mathrm{op}}
    \le M_\nu\|x-y\|^\nu,
    ~~ x,y\in\mathbb R^n.
\]
Here $I$ is the identity matrix, $\preceq$ denotes the semidefinite order,
and $\|\cdot\|$ and $\|\cdot\|_{\mathrm{op}}$ denote the Euclidean and
operator norms, respectively. The constants $\mu$ and $L$ bound the
curvature, while $M_\nu$ controls its spatial variation. The case
$\nu=1$ corresponds to Lipschitz-continuous Hessians.

Every $f\in\mathcal F_{\mu,L,M_\nu}^{n,\nu}$ has a unique minimizer
$x_\star:=\arg\min_x f(x)$. Write $f_\star:=f(x_\star)$. We apply
heavy-ball~\eqref{eq:hb} with fixed parameters $\gamma>0$ and
$\beta\in[0,1)$. Its initial state consists of a position $x_0$ and a
velocity $v_0$, and its equilibrium is $(x_\star,0)$. We measure the
initialization error by the Euclidean norm
$\|(x_0-x_\star,v_0)\|_2$.

\subsection{Uniform attraction}

We seek initialization guarantees that hold for every objective in the
prescribed class:
\begin{equation}
    \begin{gathered}
        \text{for every }f\in\mathcal F_{\mu,L,M_\nu}^{n,\nu}
        \text{ and }(x_0,v_0)\in\mathbb R^{2n},\\
        \|(x_0-x_\star,v_0)\|_2\le r
        ~\Longrightarrow~
        (x_k,v_k)\to(x_\star,0).
    \end{gathered}
    \label{eq:uniform_guarantee}
\end{equation}
Thus, $r$ measures the position and velocity error that can be tolerated
using only the class bounds and algorithm parameters. Define the
\emph{uniform attraction radius} by
\begin{equation}
    r_{\mathrm{univ}}
    :=
    \sup\bigl\{
        r\ge0:\eqref{eq:uniform_guarantee}\text{ holds}
    \bigr\}.
    \label{eq:uniform_radius}
\end{equation}
Its dependence on $n,\mu,L,M_\nu,\nu,\gamma$, and $\beta$ is usually suppressed below. An individual objective can have a larger
attraction region; $r_{\mathrm{univ}}$ describes the initialization
accuracy sufficient for the entire function class.

Our main question concerns the dependence of this radius on the damping
$d:=1-\beta$. We fix $\mu,L,\gamma$, and $\nu$ with $\gamma L<4$ and
study the limit $d\downarrow0$. In this regime, the strict quadratic
stability condition $\gamma L<2(1+\beta)$ holds for all sufficiently
small $d$. We seek a scale $s(d,M_\nu)$ and constants $c,C,d_0>0$,
depending only on the fixed parameters, such that
\begin{equation}
    c\,s(d,M_\nu)
    \le r_{\mathrm{univ}}
    \le C\,s(d,M_\nu)
    ~~ (0<d<d_0)
    \label{eq:radius_problem}
\end{equation}
for every dimension $n\ge1$ and every $M_\nu>0$. The lower bound requires
convergence for all admissible objectives and initial states in the
corresponding ball. The upper bound requires an admissible objective
with a nonconvergent initial state at a comparable distance from
equilibrium.

The dependence on $M_\nu$ follows from spatial scaling. For $R>0$,
define $f_R(x):=R^2f(x/R)$. This transformation preserves the curvature
bounds, multiplies the Hessian H\"older constant by $R^{-\nu}$, and
scales both position and velocity by $R$ without changing $\gamma$ or
$\beta$. Consequently, with all other parameters fixed,
\begin{equation}
    r_{\mathrm{univ}}(M_\nu)
    =
    M_\nu^{-1/\nu}r_{\mathrm{univ}}(1).
    \label{eq:radius_scaling}
\end{equation}
The remaining question is the sharp dependence on damping.
The uniform attraction radius specifies the required initialization
accuracy, but its definition involves the unknown minimizer $x_\star$.
To decide when to start heavy-ball, we also seek a sufficient condition
expressed through the current gradient, algorithm parameters, and class bounds. This turns the accuracy requirement into a
computable startup criterion.

\section{Convergence Guarantees}
\label{sec:energy}

Consider minimizing an objective
$f\in\mathcal F_{\mu,L,M_\nu}^{n,\nu}$ with minimizer $x_\star$.
We establish the lower bound in~\eqref{eq:radius_problem} by finding an
explicit radius for which the heavy-ball iterates converge to
$(x_\star,0)$ as required by~\eqref{eq:uniform_guarantee}.

Our approach is to construct a Lyapunov function
$\mathcal H_f(x,v)$ and a budget $\mathcal B_\nu>0$ such that the
sublevel set
$\Omega_f:=\{(x,v):\mathcal H_f(x,v)\le\mathcal B_\nu\}$
is invariant and every trajectory starting in it converges.
A ball contained in $\Omega_f$ then gives a uniform attraction-radius
bound. Evaluating the same energy at zero velocity yields a
gradient-based startup condition.

Recall $d=1-\beta$, and set
$b:=1+\beta$, $h:=\gamma L$, and $\tau:=2b-h$.
Throughout this section, we assume $\tau>0$, the strict quadratic
stability condition. We suppress the subscript $f$ on the energy.

\subsection{A Lyapunov function}

Write $g:=\nabla f(x)$. A natural discrete energy is
\[
    E(x,v)
    :=
    f(x)-f_\star
    +\frac{b}{4\gamma}\|v\|^2
    -\frac12\langle v,g\rangle.
\]
It combines objective error, motion, and a cross term that separates
damping from the error caused by varying curvature
\citep{orvieto2024}.

To examine one iteration, let $g_k:=\nabla f(x_k)$,
$v:=v_k$, and $w:=v_{k+1}=x_{k+1}-x_k$. Define the trapezoidal error
\(
    T_k
    :=
    \frac12\langle g_k+g_{k+1},w\rangle
    -\bigl(f(x_{k+1})-f(x_k)\bigr).
\)
Direct substitution of the heavy-ball update gives
\begin{equation}
    E(x_{k+1},v_{k+1})-E(x_k,v_k)
    =
    -\frac{d}{4\gamma}\|v+w\|^2-T_k.
    \label{eq:base_energy}
\end{equation}
For quadratic objectives, $T_k=0$. Under the H\"older Hessian
assumption,
\[
    |T_k|
    \le
    \frac{M_\nu}{4(\nu+2)}\|w\|^{2+\nu}.
\]
Appendix~\ref{app:energy} derives this estimate; related trapezoidal
bounds appear in \citet{marumo2025}. The difficulty in using
\eqref{eq:base_energy} is that $v+w$ can be small even when the next
displacement $w$ is large.

To introduce direct displacement dissipation, let $t>0$ be a correction
parameter and set $A:=b-(1+3\beta)t$. Consider the corrected energy
\begin{equation}
    \mathcal H(x,v)
    :=
    f(x)-f_\star
    +\frac{A}{4\gamma}\|v\|^2
    -\frac{1-t}{2}\langle v,\nabla f(x)\rangle.
    \label{eq:corrected_energy}
\end{equation}
The relation between $A$ and $t$ organizes the quadratic terms in the
one-step energy change into two forms: one involving $v+w$ and the
other involving $w$. The coefficient of $\|v+w\|^2$ becomes
$-(d-bt)/(4\gamma)$. We choose $t:=d/(2b)$, retaining half of the
original dissipation while obtaining direct control of the
displacement.

Define the averaged Hessian along the step by
$\overline H_k:=\int_0^1\nabla^2f(x_k+uw)\,du$.
Let $\mathcal H_k:=\mathcal H(x_k,v_k)$, we obtain
\begin{equation}
    \begin{aligned}
        \mathcal H_{k+1}-\mathcal H_k
        ={}&
        -\frac{d}{8\gamma}\|v+w\|^2-T_k\\
        &-\frac{t}{2\gamma}
        \langle w,(2bI-\gamma\overline H_k)w\rangle.
    \end{aligned}
    \label{eq:corrected_identity}
\end{equation}
Since $2bI-\gamma\overline H_k\succeq\tau I$, this yields
\begin{equation}
    \begin{aligned}
        \mathcal H_{k+1}-\mathcal H_k
        \le{}&
        -\frac{d}{8\gamma}\|v+w\|^2
        -\frac{\tau t}{2\gamma}\|w\|^2\\
        &+\frac{M_\nu}{4(\nu+2)}\|w\|^{2+\nu}.
    \end{aligned}
    \label{eq:energy_descent}
\end{equation}

For fixed $\mu,L,\gamma$, and $\nu$ with $\gamma L<4$, we have
$t\asymp d$ and $\tau\asymp1$ as $d\downarrow0$.
Thus, displacement dissipation dominates the nonlinear error when
$M_\nu\|w\|^\nu$ is sufficiently small relative to $d$.
This identifies the scale $(d/M_\nu)^{1/\nu}$.
To obtain an initialization guarantee, we must ensure that the energy
keeps every subsequent displacement within this scale.

\subsection{Invariance and startup}

The corrected energy controls both objective error and the next
displacement. Define
$D:=2A-h(1-t)^2$ and $N:=2\beta^2+hd/2$.
Under $\tau>0$, we have $A,N>0$ and
$D=\tau(1-t)^2+td>0$.
Completing the square gives
\begin{equation}
    \begin{aligned}
        \mathcal H(x,v)\ge{}&
        \frac{D}{2A}\bigl(f(x)-f_\star\bigr)\\
        &+\frac{A}{4\gamma}
        \left\|v-\frac{\gamma(1-t)}{A}\nabla f(x)\right\|^2.
    \end{aligned}
    \label{eq:energy_coercivity}
\end{equation}
In particular, $\mathcal H$ is nonnegative and vanishes only at
$(x_\star,0)$. Moreover,
\begin{equation}
    \|\beta v-\gamma\nabla f(x)\|^2
    \le
    \frac{4\gamma N}{D}\mathcal H(x,v).
    \label{eq:next_velocity}
\end{equation}
Hence, a sufficiently small energy budget bounds the next
displacement, which in turn makes~\eqref{eq:energy_descent}
nonpositive.

\begin{theorem}[Invariant energy region]
\label{thm:invariant}
Suppose $f\in\mathcal F_{\mu,L,M_\nu}^{n,\nu}$ and $\tau>0$.
Define
\begin{equation}
    R_\nu
    :=
    \left(
        \frac{(\nu+2)\tau t}{\gamma M_\nu}
    \right)^{1/\nu},
    ~~
    \mathcal B_\nu
    :=
    \frac{D R_\nu^2}{4\gamma N}.
    \label{eq:energy_budget}
\end{equation}
If $\mathcal H_0\le\mathcal B_\nu$, then
$\mathcal H_k\le\mathcal B_\nu$ and
$\|v_{k+1}\|\le R_\nu$ for every $k\ge0$.
Furthermore, with
$\lambda:=d\tau/[8(7+8/(\gamma\mu))]\in(0,1)$,
\(
    \mathcal H_k
    \le
    (1-\lambda)^k\mathcal H_0,
\)
and
\(
    \|x_k-x_\star\|^2
    \le
    \frac{4A}{D\mu}(1-\lambda)^k\mathcal H_0.
\)
Consequently, $(x_k,v_k)\to(x_\star,0)$.
\end{theorem}

The invariance argument combines~\eqref{eq:next_velocity} with the
definition of $R_\nu$. If $\mathcal H_k\le\mathcal B_\nu$, then
$\|w\|\le R_\nu$, so the nonlinear remainder consumes at most half of
the displacement dissipation. Therefore,
\begin{equation}
    \mathcal H_{k+1}-\mathcal H_k
    \le
    -\frac{d}{8\gamma}\|v+w\|^2
    -\frac{\tau t}{4\gamma}\|w\|^2
    \le0.
    \label{eq:closed_descent}
\end{equation}
This closes the induction. To obtain the contraction rate, use
$f(x)-f_\star\le\|\nabla f(x)\|^2/(2\mu)$ together with
$v=(v+w)-w$ and $\gamma g_k=\beta(v+w)-bw$ to bound
$\mathcal H_k$ by the dissipative terms in
\eqref{eq:closed_descent}. Comparing coefficients yields $\lambda$;
Appendix~\ref{app:energy} gives the full calculation.
The position estimate follows from~\eqref{eq:energy_coercivity},
and the velocity converges by~\eqref{eq:next_velocity}.

At zero velocity,
$\mathcal H(x,0)=f(x)-f_\star
\le\|\nabla f(x)\|^2/(2\mu)$.
The energy budget therefore gives a startup test that does not require
the minimizer or the optimal value.

\begin{corollary}[Gradient-based startup]
\label{cor:startup}
Suppose $f\in\mathcal F_{\mu,L,M_\nu}^{n,\nu}$ and the target
heavy-ball parameters satisfy $\tau>0$. If
\begin{equation}
    \|\nabla f(x)\|^2\le2\mu\mathcal B_\nu,
    \label{eq:momentum_start}
\end{equation}
then heavy-ball initialized at $(x,0)$ satisfies
Theorem~\ref{thm:invariant}. Gradient descent with step size $1/L$
reaches this condition after finitely many iterations from every
initial position.
\end{corollary}

Indeed, gradient descent converges to $x_\star$, so its gradient norm
tends to zero, while $\mathcal B_\nu>0$. This gives a concrete
procedure: run gradient descent until~\eqref{eq:momentum_start}
holds, reset the velocity to zero, and start heavy-ball with the target
parameters. Appendix~\ref{app:asymptotic} gives an explicit bound on
the warm-up length.

\paragraph{Local rate.}
\label{sec:local_rate}
Define
$q:=(\sqrt L-\sqrt\mu)/(\sqrt L+\sqrt\mu)$,
$\gamma_\star:=4/(\sqrt L+\sqrt\mu)^2$, and
$\beta_\star:=q^2$.
These parameters satisfy $\tau>0$. After startup, convergence to
$(x_\star,0)$ allows the local linearization to determine the
asymptotic rate:
\begin{equation}
    \limsup_{k\to\infty}
    \|(x_k-x_\star,v_k)\|_2^{1/k}
    \le q.
    \label{eq:startup_rate}
\end{equation}
Thus, the startup criterion provides a computable route from an
arbitrary position to the accelerated local asymptotic rate bound.
A finite warm-up does not change this root rate.
Appendix~\ref{app:asymptotic} proves~\eqref{eq:startup_rate},
including the case of repeated characteristic roots.

\subsection{Radius lower bound}

To convert the energy region into a ball around equilibrium, define
$C_H:=\max\{3L/4,\,1/(2\gamma)+L/4\}$.
The curvature bound and Cauchy--Schwarz inequality give
$\mathcal H(x,v)\le C_H\|(x-x_\star,v)\|_2^2$.
Consequently, every state satisfying
\[
    \|(x-x_\star,v)\|_2
    \le
    R_\nu\sqrt{\frac{D}{4\gamma N C_H}}
\]
belongs to the invariant region. By~\eqref{eq:uniform_radius},
\begin{equation}
    r_{\mathrm{univ}}
    \ge
    R_\nu\sqrt{\frac{D}{4\gamma N C_H}}.
    \label{eq:radius_lower_bound}
\end{equation}

\begin{corollary}[Radius lower bound]
\label{cor:radius_lower}
Fix $0<\mu<L$, $\gamma>0$ with $\gamma L<4$, and $0<\nu\le1$.
There exist constants $c,d_0>0$, depending only on
$\mu,L,\gamma$, and $\nu$, such that, for every $n\ge1$,
$M_\nu>0$, and $0<d<d_0$,
\[
    r_{\mathrm{univ}}
    \ge
    c\left(\frac{d}{M_\nu}\right)^{1/\nu}.
\]
\end{corollary}

To see the damping dependence, set $\Delta:=4-\gamma L>0$.
As $d\downarrow0$, we have
$t/d\to1/4$, $\tau\to\Delta$, $D\to\Delta$, and $N\to2$.
Substituting these limits into~\eqref{eq:radius_lower_bound}
gives the stated scale, with constants independent of dimension.

This establishes uniform convergence at distance
$c(d/M_\nu)^{1/\nu}$ from equilibrium.
Section~\ref{sec:cycles} constructs nonconvergent states at the same
order, completing the two-sided characterization.

\section{Attracting Cycles}
\label{sec:cycles}

Section~\ref{sec:energy} proves uniform convergence whenever the
initialization error is at most $c(d/M_\nu)^{1/\nu}$.
We now ask whether this scale can be enlarged.
By~\eqref{eq:uniform_radius}, a single admissible objective with a
nonconvergent initial state at distance $R$ from equilibrium implies
$r_{\mathrm{univ}}\le R$.

We construct such states on periodic orbits. The construction must
simultaneously satisfy the curvature and Hessian regularity bounds
and place the orbit at the required distance from equilibrium.
The resulting cycles are locally attracting, so nearby initial states
also approach persistent oscillations.

\begin{theorem}[Sharp attraction scale]
\label{thm:sharp_radius}
Fix $0<\mu<L$, $\gamma>0$ with $\gamma L<4$, and $0<\nu\le1$.
There exist constants $c,C,d_0>0$, depending only on
$\mu,L,\gamma$, and $\nu$, such that, for every $n\ge1$,
$M_\nu>0$, and $0<d<d_0$,
\[
    c\left(\frac{d}{M_\nu}\right)^{1/\nu}
    \le
    r_{\mathrm{univ}}
    \le
    C\left(\frac{d}{M_\nu}\right)^{1/\nu}.
\]
The same bounds hold when the objective class is restricted to
$C^\infty$ functions. The upper bound is witnessed by locally
attracting periodic orbits.
\end{theorem}

\paragraph{Damping and initialization.}
The theorem quantifies the relation between damping and initialization
accuracy. In the stated small-damping regime, guaranteeing convergence
throughout a ball of radius $r$ requires $d\gtrsim M_\nu r^\nu$,
and a sufficiently large constant multiple of $M_\nu r^\nu$
is sufficient. More accurate initialization therefore permits smaller
damping, with the exponent determined by Hessian regularity.

The lower bound is Corollary~\ref{cor:radius_lower}.
We prove the upper bound in one dimension and then embed the
construction into higher dimensions.

\subsection{Cycle interpolation}

Consider first the undamped iteration, $\beta=1$, on a quadratic with
curvature $2(1-\cos\theta)/\gamma$. Its recurrence is
$x_{j+1}=2\cos\theta\,x_j-x_{j-1}$, which admits oscillations
$x_j:=R\cos(j\theta-\phi)$.

The map $\theta\mapsto2(1-\cos\theta)$ is continuous from $(0,\pi)$
onto $(0,4)$. Since $0<\gamma\mu<\gamma L<4$ and rational multiples
of $2\pi$ are dense, we can choose coprime positive integers $m,K$
such that $\theta:=2\pi m/K$ satisfies
\begin{equation}
    0<\theta<\pi,
    ~~
    \gamma\mu<2(1-\cos\theta)<\gamma L.
    \label{eq:resonance}
\end{equation}
This produces a periodic quadratic oscillation with curvature strictly
inside the prescribed interval.

Now restore damping, $\beta=1-d$. Set
$a:=1-\cos\theta$, $s:=\sin\theta$, and $u_j:=j\theta-\phi$.
For the positions $x_j=R\cos u_j$ to remain a heavy-ball cycle, the
objective must satisfy
\begin{equation}
    \begin{aligned}
        \gamma f'(R\cos u_j)
        &=b x_j-x_{j+1}-\beta x_{j-1}\\
        &=R\bigl(ba\cos u_j+ds\sin u_j\bigr),
    \end{aligned}
    \label{eq:cycle_interpolation}
\end{equation}
with indices interpreted modulo $K$.
Thus, constructing a cycle amounts to interpolating these prescribed
gradient values while preserving the function-class bounds.

The first term in~\eqref{eq:cycle_interpolation} comes from a quadratic
of curvature $ba/\gamma$. The second compensates for damping and has
magnitude of order $dR$. Realizing this correction over a spatial
scale $R$ requires a Hessian perturbation of order $d$ and a Hessian
H\"older constant of order $d/R^\nu$.
Matching this constant to $M_\nu$ leads to
$R\asymp(d/M_\nu)^{1/\nu}$.

\subsection{Smooth realization}

Choose $\phi:=\pi/(2K)$ and define $z_j:=\cos u_j$.
These nodes are distinct. Fix a function $H\in C^\infty(\mathbb R)$
satisfying $H(z_j)=\sin u_j$, $\|H\|_\infty\le1$, and such that
$H$ is constant in a neighborhood of each node.
One obtains such a function by sorting the nodes, joining consecutive
prescribed values with flat smooth transitions, and extending
constantly beyond the outermost nodes.

For a scalar function $p$, define its H\"older seminorm by
$[p]_\nu:=\sup_{x\ne y}|p(x)-p(y)|/|x-y|^\nu$.
Set $K_{H,\nu}:=[H']_\nu\in(0,\infty)$.
This constant depends on the fixed interpolation profile and $\nu$,
but not on $d$ or $M_\nu$.
Define
\begin{equation}
    \begin{aligned}
        R_c
        &:=
        \left(\frac{dsK_{H,\nu}}{\gamma M_\nu}\right)^{1/\nu},\\
        f'(x)
        &:=
        \frac{ba}{\gamma}x
        +\frac{dR_cs}{\gamma}H(x/R_c).
    \end{aligned}
    \label{eq:cycle_construction}
\end{equation}
The second line defines $f$ up to an additive constant.

Differentiation gives
$\gamma f''(x)=(2-d)a+dsH'(x/R_c)$ and
\[
    [f'']_\nu
    =
    \frac{dsK_{H,\nu}}{\gamma R_c^\nu}
    =
    M_\nu.
\]
Moreover, the curvature bounds hold whenever
\[
    d\bigl(a+s\|H'\|_\infty\bigr)
    \le
    \min\{2a-\gamma\mu,\ \gamma L-2a\}.
\]
Both margins are positive, so this condition
holds for all sufficiently small $d$.
Hence, $f\in\mathcal F_{\mu,L,M_\nu}^{1,\nu}\cap C^\infty$.

The positions $x_j:=R_cz_j$ satisfy
\eqref{eq:cycle_interpolation} exactly and therefore form a
$K$-periodic orbit. Since $H$ is flat near every node,
$f''(x_j)=ba/\gamma$. In position--velocity coordinates, the
heavy-ball map has the same Jacobian at every point of the cycle:
\[
    J
    :=
    \begin{pmatrix}
        1-ba & \beta\\
        -ba & \beta
    \end{pmatrix}.
\]
Its characteristic polynomial is
$p(\zeta):=\det(\zeta I-J)
=\zeta^2-b\cos\theta\,\zeta+\beta$.
We have
$p(1)=b(1-\cos\theta)>0$,
$p(-1)=b(1+\cos\theta)>0$, and $0\le\beta<1$.
These conditions place both roots strictly inside the unit disk.
The return map has derivative $J^K$, so the cycle is locally
attracting.

It remains to measure its distance from equilibrium.
Strong convexity and~\eqref{eq:cycle_construction} imply
\[
    |x_\star|
    \le
    \frac{|f'(0)|}{\mu}
    \le
    \frac{dR_cs}{\gamma\mu}.
\]
Every cycle position satisfies $|x_j|\le R_c$, and its velocity
$v_j:=x_j-x_{j-1}$ satisfies $|v_j|\le2R_c$.
Therefore,
\begin{equation}
    \begin{aligned}
        \|(x_j-x_\star,v_j)\|_2
        &\le
        R_c\sqrt{\left(1+\frac{ds}{\gamma\mu}\right)^2+4}\\
        &\le
        C\left(\frac{d}{M_\nu}\right)^{1/\nu}.
    \end{aligned}
    \label{eq:cycle_state_bound}
\end{equation}
These states do not converge to equilibrium, which proves the
one-dimensional upper bound.

For $n>1$, define
$F(x):=f(x_1)+(\mu/2)\sum_{\ell=2}^n x_\ell^2$.
The function $F$ has the same curvature and Hessian regularity bounds,
and the cycle persists with all additional coordinates and velocities
equal to zero. The added quadratic modes are strictly stable, so the
cycle remains locally attracting in the full state space.
This completes the upper bound in
Theorem~\ref{thm:sharp_radius}.
Appendix~\ref{app:holder} supplies the construction details.
Appendix~\ref{app:cycle_size} gives a complementary lower bound on
the diameter of every nonconstant periodic orbit.

\subsection{An explicit three-cycle}
\label{sec:explicit_cycle}

For Lipschitz Hessians, write $M:=M_1$.
The choice $\theta=2\pi/3$ gives an undamped quadratic curvature
$3/\gamma$. Taking $\mu=1$, $L=7/2$, and $\gamma=1$ places this
curvature strictly inside the admissible interval and gives a
three-cycle with explicit constants. This parameter choice also
underlies the numerical experiments.

Let $0<d\le1/10$. Choose a smooth profile $H$ that is flat near
$-1,0,1$ and satisfies
$H(-1)=H(1)=1/2$, $H(0)=-1$,
$\|H'\|_\infty\le3$, and $\|H''\|_\infty\le9$.
Define $R_c:=9d/M$ and
\[
    f'(x)
    :=
    \left(3-\frac32d\right)x+dR_cH(x/R_c).
\]
This defines a smooth function in
$\mathcal F_{1,7/2,M}^{1,1}$ with the exact cycle
$-R_c\to0\to R_c\to-R_c$.
At each cycle point, $f''=3b/2$, and the corresponding return map is
strictly stable.

The minimizer satisfies $0<x_\star\le dR_c$.
At the cycle phase $(x,v)=(0,R_c)$, the distance from equilibrium is
at most $R_c\sqrt{1+d^2}<10d/M$.
Combining this upper bound with the energy region gives the following
fully explicit specialization.

\begin{corollary}[Explicit radius bounds]
\label{cor:explicit_radius}
Let $\mu=1$, $L=7/2$, $\gamma=1$, $M>0$, and
$9/10\le\beta<1$. For every dimension $n\ge1$,
\[
    \frac{1-\beta}{40M}
    \le
    r_{\mathrm{univ}}
    \le
    \frac{10(1-\beta)}{M}.
\]
The bounds also hold on the $C^\infty$ subclass.
Every initial state in the ball of radius $(1-\beta)/(40M)$
satisfies
\[
    \mathcal H_k
    \le
    \left(1-\frac{1-\beta}{400}\right)^k\mathcal H_0.
\]
\end{corollary}

Appendix~\ref{app:sharp} verifies the profile construction and the
constants. The contraction estimate follows from
$\tau=1/2-2d\ge3/10$ and
$\lambda=d\tau/120\ge d/400$.

\subsection{Global non-acceleration}
\label{sec:global_obstruction}

For $\kappa:=L/\mu\ge24$, convergence over the entire function class
from every initial state requires the fixed parameters $\gamma>0$
and $0\le\beta<1$ to satisfy $\gamma L<12(1-\beta)$.
We establish this necessary condition using cycle interpolation and
quadratic stability, then use a low-curvature quadratic to bound the
best uniform global rate.

Choose an integer $K\ge3$, set $\theta:=2\pi/K$ and
$\phi:=\theta/4$, and consider the unit-amplitude positions
$x_j:=\cos(j\theta-\phi)$.
Prescribe gradient values
$g_j:=(b x_j-x_{j+1}-\beta x_{j-1})/\gamma$
as in~\eqref{eq:cycle_interpolation}.
With $a_K:=b(1-\cos\theta)$, every pair of distinct nodes satisfies
\[
    \left|
        \gamma\frac{g_i-g_j}{x_i-x_j}-a_K
    \right|
    <4d.
\]
Thus, the period controls the center of the interval containing all
normalized secant slopes.

Suppose $12d\le h<2b$. Choose the smallest $K\ge3$ such that
$a_K\le2h/3$. Then $3h/8<a_K\le2h/3$.
The normalized secant slopes lie strictly between $h/24$ and $h$,
and hence between $\gamma\mu$ and $\gamma L$ when $\kappa\ge24$.
Smooth interpolation realizes these gradient data with curvature in
$[\mu,L]$ and produces a locally attracting cycle.
Spatial dilation then enforces any prescribed positive Hessian
H\"older bound.

These cycles exclude the parameter region $12d\le h<2b$ from any
global convergence guarantee. If $h\ge2b$, the quadratic of curvature
$L$ has a nondecaying mode. Global convergence therefore requires
$h<2b$ and $h<12d$. In this remaining region, the quadratic of
curvature $\mu$ yields the slow-convergence lower bound.

To state the rate consequence, let
$\mathcal G:=\mathcal F_{\mu,L,M_\nu}^{1,\nu}$ and
$z_k:=(x_k-x_\star,v_k)$. Define the best uniform global factor by
\begin{equation}
    \rho_\star
    :=
    \inf_{\substack{\gamma>0\\0\le\beta<1}}
    \limsup_{k\to\infty}
    \left(
        \sup_{\substack{f\in\mathcal G\\z_0\ne0}}
        \frac{\|z_k\|_2}{\|z_0\|_2}
    \right)^{1/k}.
    \label{eq:global_factor}
\end{equation}
The inner supremum takes the worst case over objectives and
initial states before the asymptotic limit; the outer infimum
optimizes the fixed parameter pair.
For $\kappa\ge24$,
\begin{equation}
    \frac{\kappa-12}{\kappa+12}
    \le
    \rho_\star
    \le
    \frac{\kappa-1}{\kappa+1}.
    \label{eq:global_nonacceleration}
\end{equation}
The upper bound is attained as a guarantee by gradient descent,
and the lower bound remains valid on the $C^\infty$ subclass.
Consequently, $1-\rho_\star=\Theta(\kappa^{-1})$:
fixed-parameter heavy-ball cannot improve the condition-number order
of the uniform global rate.
This gives an elementary scalar proof of the known global
non-acceleration
\citep{goujaud2025,fotopoulos2026}.
Appendix~\ref{app:nonaccel} contains the full argument.

The factor~\eqref{eq:global_factor} concerns heavy-ball using one
parameter pair throughout and takes the worst case over objectives
and initial states before the asymptotic limit.
The startup procedure in Corollary~\ref{cor:startup} first uses
gradient descent to select the initial state of the heavy-ball phase.
With quadratic-optimal parameters,
\eqref{eq:startup_rate} then gives an accelerated asymptotic rate
bound for each resulting trajectory.
The total gradient cost at a prescribed accuracy includes both
warm-up and transient iterations, as measured in the numerical
experiments.

\section{Numerical Experiments}
\label{sec:numerics}

We examine the invariant energy region, the spatial scale of periodic
failure, and the gradient-based startup procedure. The final experiment
uses quadratic-optimal parameters to study the local rate after startup
and its contribution to the total gradient cost.

The first three experiments use a $C^3$ objective obtained from the
three-cycle construction in Section~\ref{sec:explicit_cycle}, with
$\mu=\gamma=M_\nu=1$ and $L=7/2$:
\begin{equation}
    f'(x)
    :=
    \left(3-\frac32d\right)x+dR_cH_0(x/R_c).
    \label{eq:numerical_objective}
\end{equation}
Here $H_0(u):=-1+\frac32S(\min\{|u|,1\})$,
$S(s):=6s^5-15s^4+10s^3$, and
$R_c:=(dK_\nu)^{1/\nu}$, where
$K_\nu:=(45/8)^{1-\nu}(5\sqrt3)^\nu$
is a global upper bound on $[H_0']_\nu$.
These objectives satisfy the prescribed curvature and H\"older bounds
for all tested damping values.

Computations use double precision. Phase scans and startup comparisons
use normalized position--velocity coordinates. For the three-cycle
experiments, attractor classification requires six consecutive states
within distance $10^{-7}R_c$ of the equilibrium or the cycle.
Appendix~\ref{app:numerical_protocol} provides the protocols and
additional results.

\subsection{Energy and attraction}

We first take $\nu=1$, so $M:=M_1=1$.
Figure~\ref{fig:dynamics_scales}(a) shows the coexisting equilibrium
and cycle attractors at $\beta=0.95$.
On a $101\times101$ phase-space grid, 2643 trajectories approach the
equilibrium and 7558 approach the cycle under the finite-time
classification rule, with no unresolved states after 1800 updates.

To examine Theorem~\ref{thm:invariant} and
Corollary~\ref{cor:explicit_radius}, we sample 2048 initial states
inside the guaranteed ball of radius $d/(40M)$ across
$\beta\in\{0.9,0.95,0.99,0.999\}$.
The energy-budget and velocity bounds hold throughout 240 updates
to numerical precision.
Figure~\ref{fig:dynamics_scales}(b) shows a representative trajectory
at $\beta=0.95$: after 320 updates, its energy is
$7.48\times10^{-8}$ of its initial value and remains below the
contraction envelope. The envelope uses the sharper factor
$\widehat\lambda$ from \eqref{app:eq:lambda_sharp}.

\begin{figure*}[ht]
    \centering
    \includegraphics[width=0.8\textwidth]
        {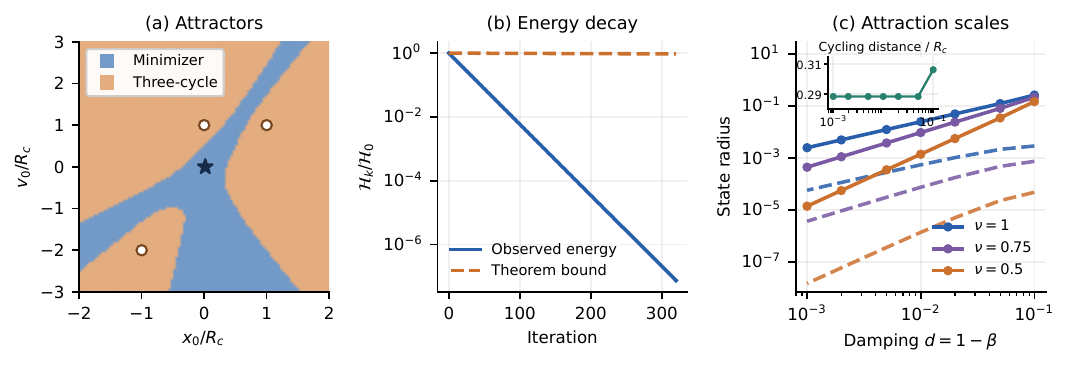}
    \caption{\small
    Attraction and initialization scales.
    (a) Finite-time phase classification at $\beta=0.95$, $\nu=1$.
    The star marks the minimizer and circles mark cycle phases.
    (b) Normalized energy and the envelope
    $(1-\widehat\lambda)^k$.
    (c) Smallest sampled cycling distance (solid) and the
    guaranteed radius from \eqref{eq:radius_lower_bound} (dashed).
    The inset gives cycling distance divided by $R_c$.
    The H\"older cases share normalized scans and use their
    respective physical spatial scales.
    }
    \label{fig:dynamics_scales}
\end{figure*}

\subsection{Cycling scale}

The construction in Section~\ref{sec:cycles} places a periodic orbit
at scale $R_c$. We examine how closely its attraction region
approaches the equilibrium by scanning nearby initial states.
Across seven damping values from $10^{-3}$ to $10^{-1}$, the
smallest sampled cycling distance ranges from $0.2885R_c$ to
$0.3064R_c$. All 43,008 scanned states are classified.
At $d=0.01$, doubling both angular and radial resolution changes
this distance by $1.23\%$.

Thus, over the tested damping range, the cycle attracts sampled
initial states at approximately $0.29R_c$ from equilibrium.
Figure~\ref{fig:dynamics_scales}(c) displays these distances alongside
the uniform guarantees from \eqref{eq:radius_lower_bound}.
Its inset shows that the normalized distance varies little with $d$.
The normalized dynamics are independent of $\nu$, so the curves for
$\nu\in\{1,3/4,1/2\}$ use the same scans with
$R_c=(dK_\nu)^{1/\nu}$.
The observed attraction geometry therefore persists under the
spatial rescalings used in the sharp-radius construction.

\subsection{Gradient-based startup}
\label{sec:numerics_startup}

We next test Corollary~\ref{cor:startup} for prescribed target
parameters $\gamma=1$ and $\beta\in\{0.9,0.95,0.99\}$.
We compare gradient descent (GD), direct heavy-ball (HB), and GD
followed by HB on 101 equally spaced initial positions in
$[-2R_c,2R_c]$, all with zero initial velocity.
GD uses step size $1/L$. The hybrid switches at the first state
satisfying \eqref{eq:momentum_start}, resets the velocity to zero,
and uses the target HB parameters thereafter.
The switch test reuses the current gradient, so each update costs
one gradient call.

Across the three momentum values and the startup budgets for
$\nu\in\{1,3/4,1/2\}$, all 909 hybrid runs reach the equilibrium
tolerance, with switching after zero to seven GD updates.
Every post-switch trajectory satisfies the energy and velocity
bounds to numerical precision.
Table~\ref{tab:startup} reports the Lipschitz cases.
\begin{figure*}[ht]
    \centering
    \includegraphics[width=0.8\textwidth]
        {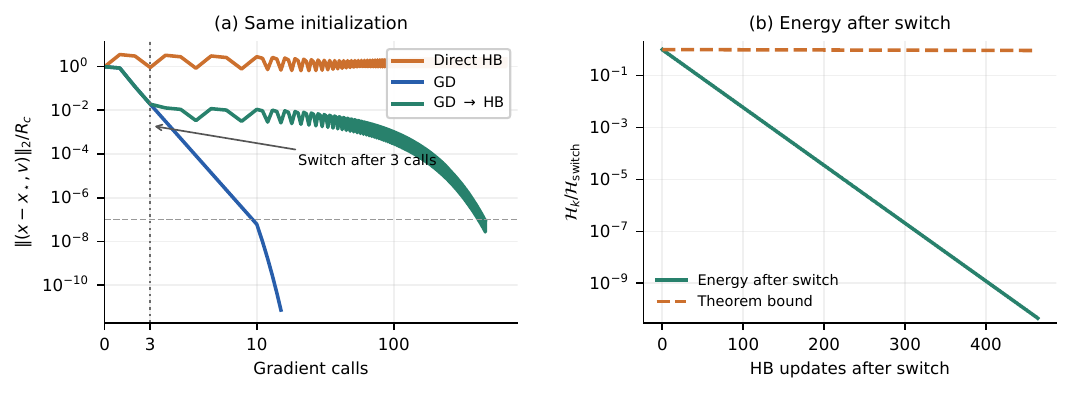}
    \caption{\small
    Startup from $(x_0,v_0)=(R_c,0)$ at
    $\gamma=1$, $\beta=0.95$, and $\nu=1$.
    (a) State error versus gradient calls; the vertical line marks
    the switch and the horizontal line marks the tolerance.
    The horizontal axis is linear up to 10 calls and logarithmic
    thereafter.
    (b) Post-switch energy and its envelope using
    $\widehat\lambda$ from \eqref{app:eq:lambda_sharp}.
    All methods use the same state-error tolerance and six-state
    confirmation rule.
    }
    \label{fig:startup}
\end{figure*}

\begin{table}[t]
    \centering
    \small
    \setlength{\tabcolsep}{3pt}
    \begin{tabular}{ccccc}
        \toprule
        $\beta$ & HB cycles & GD warm-up & GD calls & Hybrid calls\\
        \midrule
        $0.90$ & $83/101$ & $1$--$5$ & $16$ & $213$\\
        $0.95$ & $83/101$ & $1$--$4$ & $15$ & $426$\\
        $0.99$ & $84/101$ & $0$--$4$ & $15$ & $2239$\\
        \bottomrule
    \end{tabular}
    \caption{\small
    Startup on the three-cycle family, $\nu=1$.
    GD and the hybrid reach the tolerance from all 101 starts.
    Warm-up gives the range of GD updates before switching.
    The final two columns give median total gradient calls,
    including the six-state stopping confirmation.
    }
    \label{tab:startup}
\end{table}

Figure~\ref{fig:startup} illustrates
$(x_0,v_0)=(R_c,0)$ at $\beta=0.95$ and $\nu=1$.
The hybrid switches after three GD calls and reaches the stopping
criterion after 467 total calls. GD alone takes 15 calls, while
direct HB approaches the cycle.
For these target parameters, the startup test eliminates periodic
failure, while GD has the lower total gradient cost.

\subsection{Local rate}
\label{sec:numerics_local}

To examine \eqref{eq:startup_rate}, we use nonquadratic objectives
with an exactly quadratic neighborhood of the minimizer.
For $\kappa\in\{4,16,64,256\}$, set
$a_\kappa:=15(\kappa-1)/8$ and define $f_\kappa$ by
\begin{equation}
    f_\kappa''(x)
    :=
    1+(\kappa-1)
    S\!\left(
        \operatorname{clip}\!\left(\frac{|x|}{a_\kappa}-1,0,1\right)
    \right),
    \label{eq:numerical_local_objective}
\end{equation}
with $f_\kappa(0)=f_\kappa'(0)=0$.
Here $\operatorname{clip}(z,0,1):=\min\{1,\max\{0,z\}\}$.
These objectives satisfy $\mu=M=1$, $L=\kappa$, and
$f_\kappa(x)=x^2/2$ on $|x|\le a_\kappa$.

Starting from $(2a_\kappa,0)$, the hybrid uses GD with step $1/L$
until \eqref{eq:momentum_start} holds and then switches to the
quadratic-optimal parameters from Section~\ref{sec:energy}.
All four switches occur inside the quadratic neighborhood.
Writing $T$ for the switching time, the subsequent positions obey
$x_{T+j}=x_T[1+(1-q)j]q^j$.
The computed position--velocity trajectories agree with this
repeated-root solution to relative error below $3.1\times10^{-14}$.
The formula exhibits the linear prefactor and asymptotic root factor $q$.
Figure~\ref{app:fig:polyak_startup} shows the trajectories and
local decay.

Table~\ref{tab:polyak_startup} separates warm-up and HB costs and
compares their sum with GD using step
$2/(L+\mu)$. Each method uses the state-error tolerance
$10^{-10}a_\kappa$ and six consecutive confirmations.
The hybrid has the lower total cost for these initial states,
although warm-up accounts for most of its cost at larger $\kappa$.
Appendix~\ref{app:numerical_protocol} also reports GD with step
$1/L$ and the additional initial positions
$x_0/a_\kappa\in\{1.5,3\}$.

\begin{table}[ht]
    \centering
    \small
    \setlength{\tabcolsep}{3pt}
    \begin{tabular}{ccccc}
        \toprule
        $\kappa$ & Warm-up & HB phase & Hybrid total & GD total\\
        \midrule
        $4$   & $8$    & $28$  & $36$   & $51$\\
        $16$  & $66$   & $49$  & $115$  & $190$\\
        $64$  & $399$  & $85$  & $484$  & $742$\\
        $256$ & $2129$ & $148$ & $2277$ & $2953$\\
        \bottomrule
    \end{tabular}
    \caption{\small
    Gradient calls from $(2a_\kappa,0)$.
    The hybrid uses GD with step $1/L$ before the gradient test
    and quadratic-optimal HB afterward.
    The standalone GD baseline uses step $2/(L+\mu)$.
    All totals include the stopping confirmation.
    }
    \label{tab:polyak_startup}
\end{table}

\section{Conclusion}
\label{sec:conclusion}

We establish the sharp uniform attraction scale of heavy-ball
by matching explicit convergence guarantees with smooth
attracting cycles, whose construction also yields an elementary
proof of global non-acceleration.
The energy bound gives a gradient-based startup rule that,
after finite GD warm-up, guarantees convergence and an
accelerated local asymptotic rate at quadratic-optimal parameters.
The sharp scaling fixes the curvature bounds and step size,
leaving their joint interaction with vanishing damping unresolved.
The startup rule requires prescribed curvature and Hessian
regularity bounds, and its local rate guarantee does not establish
an accelerated total gradient complexity.
Developing adaptive startup rules with end-to-end complexity
guarantees is a natural next step.

\section*{AI Use Statement}
ChatGPT 5.6 Sol and 6 Astra assisted with checking mathematical proofs, editing code for numerical examples, polishing the manuscript, and conducting a simulated peer review. The authors take full responsibility for all content in this paper.

\clearpage
\begingroup
\raggedbottom
\bibliographystyle{plainnat}
\bibliography{references}
\clearpage
\endgroup
\clearpage
\onecolumn
\appendix
\numberwithin{equation}{section}

\section{Related Work}
\label{sec:related_work}

\paragraph{Acceleration and stability.}
Heavy-ball and Nesterov's accelerated schemes provide two classical
approaches to momentum-based acceleration
\citep{polyak1964,nesterov2018}.
For quadratic objectives, robust-control arguments establish
heavy-ball's worst-case asymptotic optimality within a multistep
framework \citep{ugrinovskii2023}, while discrete Lyapunov
constructions explain accelerated quadratic dynamics
\citep{orvieto2024}.
Beyond quadratics, global convergence is available under suitable
parameter restrictions \citep{ghadimi2015}, and local asymptotic
acceleration extends to objectives satisfying the
Polyak--{\L}ojasiewicz inequality \citep{kassing2026}.
Differential-equation models clarify the roles of damping and
gradient correction \citep{su2016,shi2022}, while integral quadratic
constraints provide systematic tools for certifying convergence
rates \citep{lessard2016,fazlyab2018}.
We complement these rate analyses by quantifying the initialization
accuracy needed for convergence uniformly over a class with
prescribed Hessian regularity.

\paragraph{Cycles and non-acceleration.}
Periodic counterexamples show that stability on quadratics does not
ensure global convergence \citep{lessard2016}.
Smooth strongly convex interpolation \citep{taylor2017} supports a
constructive search for such failures \citep{goujaud2023}.
Using cyclic trajectories, \citet{goujaud2025} rule out accelerated
worst-case rates for fixed-parameter heavy-ball on general smooth
strongly convex objectives.
The one-dimensional question was subsequently isolated in
\citet{riddles2025} and resolved through projected cycles in the
research note of \citet{fotopoulos2026}.
For smooth convex objectives, \citet{ma2026} establish a separate
limitation for predetermined time-varying parameters.
Here, we construct smooth attracting cycles at controlled distances
from equilibrium under prescribed Hessian regularity bounds.
Their distance matches our convergence guarantee, identifying the
sharp initialization scale.
The same scalar interpolation also yields an elementary proof of
global non-acceleration.

\paragraph{Restart and startup.}
Adaptive restart uses function or gradient information to control
oscillations in accelerated schemes \citep{odonoghue2015}.
For nonconvex optimization, restarted accelerated gradient and
heavy-ball methods admit improved stationarity complexity bounds
\citep{li2023restart}, including parameter-free and
Hessian-regularity-adaptive variants \citep{marumo2024,marumo2025}.
Other approaches modify the iteration through over-relaxation to
obtain global acceleration \citep{wei2025}, or use curve searches to
globalize heavy-ball-type steps \citep{donnini2026}.
In our strongly convex setting, a gradient-based test determines
when a finite gradient-descent warm-up has reached an explicit
energy region.
We then reset the velocity and keep the target heavy-ball
parameters fixed.
This connects a verifiable startup condition with the subsequent
local asymptotic rate.

\section{Energy Guarantees}
\label{app:energy}

We prove the estimates used in Section~\ref{sec:energy},
Theorem~\ref{thm:invariant}, Corollary~\ref{cor:radius_lower}, and
the scaling identity~\eqref{eq:radius_scaling}.
Throughout the energy analysis,
$f\in\mathcal F_{\mu,L,M_\nu}^{n,\nu}$,
$0\le\beta<1$, and $\gamma L<2(1+\beta)$.
We retain $b:=1+\beta$, $d:=1-\beta$, $h:=\gamma L$, and
$\tau:=2b-h>0$.
All bounds below are independent of dimension.

\subsection{Energy identities}

Fix an iteration and write $x:=x_k$, $v:=v_k$,
$w:=v_{k+1}$, $g:=\nabla f(x)$, and $g^+:=\nabla f(x+w)$.
Then $\gamma g=\beta v-w$.
Define
$\overline H:=\int_0^1\nabla^2f(x+uw)\,\mathrm du$.
The fundamental theorem of calculus gives
$g^+-g=\overline H w$, with
$\mu I\preceq\overline H\preceq LI$.
For the trapezoidal error
$T:=\tfrac12\langle g+g^+,w\rangle-[f(x+w)-f(x)]$,
the base energy satisfies
\[
\begin{aligned}
    E(x+w,w)-E(x,v)
    &=\frac12\langle g,v+w\rangle
      +\frac b{4\gamma}(\|w\|^2-\|v\|^2)-T\\
    &=\frac1{4\gamma}
      \bigl[2\langle\beta v-w,v+w\rangle
            +b(\|w\|^2-\|v\|^2)\bigr]-T\\
    &=-\frac d{4\gamma}\|v+w\|^2-T.
\end{aligned}
\]
This proves~\eqref{eq:base_energy}.

To derive the correction, initially allow $t$ to be arbitrary and set
$A_t:=b-(1+3\beta)t$. Define
\[
    \mathcal H_t(x,v)
    :=f(x)-f_\star+\frac{A_t}{4\gamma}\|v\|^2
      -\frac{1-t}{2}\langle v,g\rangle.
\]
Substituting the definition of $T$ and $g^+=g+\overline H w$ gives
\[
\begin{aligned}
  4\gamma\bigl[\mathcal H_t(x+w,w)-\mathcal H_t(x,v)+T\bigr]
  ={}&2\langle\beta v-w,(1-t)v+(1+t)w\rangle\\
     &+A_t(\|w\|^2-\|v\|^2)
       +2\gamma t\langle w,\overline H w\rangle.
\end{aligned}
\]
The coefficients of $\|v\|^2$, $\langle v,w\rangle$, and
$\|w\|^2$ in the first two terms are, respectively,
$-d+bt$, $2(-d+bt)$, and $-d-3bt$.
Thus the choice of $A_t$ collects these terms into
$-(d-bt)\|v+w\|^2-4bt\|w\|^2$, yielding
\begin{equation}
    \mathcal H_t(x+w,w)-\mathcal H_t(x,v)
    =-\frac{d-bt}{4\gamma}\|v+w\|^2
     -\frac t{2\gamma}
      \langle w,(2bI-\gamma\overline H)w\rangle-T.
    \label{app:eq:general_energy}
\end{equation}
For $0<t<d/b$, both quadratic terms dissipate energy.
From now on, fix $t:=d/(2b)$, $A:=A_t$, and
$\mathcal H:=\mathcal H_t$, as in~\eqref{eq:corrected_energy}.
Then $d-bt=d/2$, so~\eqref{app:eq:general_energy} becomes
\eqref{eq:corrected_identity}.

For the remainder estimate, define $\phi(u):=f(x+uw)$ on $[0,1]$.
Integration by parts and reflection about $1/2$ give
\[
\begin{aligned}
    T
    &=\int_0^1(u-\tfrac12)\phi''(u)\,\mathrm du\\
    &=\int_{1/2}^1
       (u-\tfrac12)[\phi''(u)-\phi''(1-u)]\,\mathrm du.
\end{aligned}
\]
Because
$\phi''(u)=\langle w,\nabla^2f(x+uw)w\rangle$,
the Hessian assumption implies
$|\phi''(u)-\phi''(1-u)|
\le M_\nu\|w\|^{2+\nu}|2u-1|^\nu$.
Therefore,
\begin{equation}
    |T|
    \le M_\nu\|w\|^{2+\nu}
       \int_{1/2}^1(u-\tfrac12)(2u-1)^\nu\,\mathrm du
    =\frac{M_\nu}{4(\nu+2)}\|w\|^{2+\nu}.
    \label{app:eq:holder_remainder}
\end{equation}
Combining~\eqref{eq:corrected_identity},
\eqref{app:eq:holder_remainder}, and
$2bI-\gamma\overline H\succeq\tau I$
proves~\eqref{eq:energy_descent}.

\subsection{Energy and state bounds}

Recall $D:=2A-h(1-t)^2$ and $N:=2\beta^2+hd/2$.
The coefficient identities
\[
    A=\frac{1+2\beta+5\beta^2}{2b},
    ~~
    D=\tau(1-t)^2+td,
    ~~
    A-2\beta(1-t)=\frac d2
\]
show that $A,D>0$. Also $N>0$, since its first term is positive
when $\beta>0$ and its second term is $h/2>0$ when $\beta=0$.
We have $0<t\le1/2$ and $A\le b<2$.

For any $x$, set $\delta f:=f(x)-f_\star$ and $g:=\nabla f(x)$.
Smoothness, evaluated at $x-g/L$, implies
$f(x-g/L)\le f(x)-\|g\|^2/(2L)$.
Since $f_\star\le f(x-g/L)$, it follows that
$\|g\|^2\le2L\delta f$.
Completing the square in the corrected energy gives
\[
    \mathcal H(x,v)
    =\delta f-\frac{\gamma(1-t)^2}{4A}\|g\|^2
     +\frac A{4\gamma}
       \left\|v-\frac{\gamma(1-t)}A g\right\|^2.
\]
Substitution of the gradient bound proves
\eqref{eq:energy_coercivity}.
In particular, $\mathcal H(x,v)\ge0$, and equality forces
$\delta f=0$ and $v=0$, hence $(x,v)=(x_\star,0)$.
Strong convexity further gives
\begin{equation}
    \mathcal H(x,v)
    \ge\frac{D\mu}{4A}\|x-x_\star\|^2.
    \label{app:eq:position_energy}
\end{equation}

To bound the next velocity, use $\delta f\ge\|g\|^2/(2L)$
once more:
\[
    \mathcal H(x,v)
    \ge\frac1{4\gamma}
    \begin{bmatrix}v\\\gamma g\end{bmatrix}^{\!\top}
    (Q\otimes I)
    \begin{bmatrix}v\\\gamma g\end{bmatrix},
    ~~
    Q:=\begin{pmatrix}
        A&-(1-t)\\
        -(1-t)&2/h
    \end{pmatrix}.
\]
The matrix $Q$ is positive definite because its leading diagonal
entry is positive and $\det Q=D/h>0$. Its inverse is
\[
    Q^{-1}
    =\frac1D
      \begin{pmatrix}
          2&h(1-t)\\
          h(1-t)&hA
      \end{pmatrix}.
\]
For $r:=(\beta,-1)$, the coefficient identity above gives
\[
    rQ^{-1}r^\top
    =\frac{2\beta^2+h[A-2\beta(1-t)]}{D}
    =\frac ND.
\]
Applying Cauchy--Schwarz in the $Q$-inner product to each
coordinate and summing proves
\[
    \|\beta v-\gamma g\|^2
    \le\frac ND
       \begin{bmatrix}v\\\gamma g\end{bmatrix}^{\!\top}
       (Q\otimes I)
       \begin{bmatrix}v\\\gamma g\end{bmatrix}
    \le\frac{4\gamma N}{D}\mathcal H(x,v),
\]
which is~\eqref{eq:next_velocity}.

\subsection{Invariance and rate}

We now prove Theorem~\ref{thm:invariant}.
Let $R_\nu$ and $\mathcal B_\nu$ be defined by
\eqref{eq:energy_budget}.
If $\mathcal H_k\le\mathcal B_\nu$, then
\eqref{eq:next_velocity} gives $\|w\|\le R_\nu$ and
\[
    \frac{M_\nu}{4(\nu+2)}\|w\|^\nu
    \le\frac{M_\nu R_\nu^\nu}{4(\nu+2)}
    =\frac{\tau t}{4\gamma}.
\]
The descent estimate~\eqref{eq:energy_descent} consequently implies
\eqref{eq:closed_descent} and
$\mathcal H_{k+1}\le\mathcal H_k\le\mathcal B_\nu$.
Induction from $\mathcal H_0\le\mathcal B_\nu$
proves the invariant energy budget and the displacement bound
for every $k\ge0$.

To quantify the decrease, define
$U:=A+1-t$, $V:=2/(\gamma\mu)+1-t$,
$P:=U+\beta^2V$, and $Q_0:=U+b^2V$.
Minimizing the strong-convexity lower model of $f$ at $x$ gives
$\delta f\le\|g\|^2/(2\mu)$.
Young's inequality then implies
\[
    \mathcal H(x,v)
    \le\frac1{4\gamma}
       \bigl(U\|v\|^2+V\|\gamma g\|^2\bigr).
\]
Set $s:=v+w$ along the iteration.
Since $v=s-w$ and $\gamma g=\beta s-bw$, we obtain
\begin{equation}
    \mathcal H_k
    \le\frac1{2\gamma}
       \bigl(P\|s\|^2+Q_0\|w\|^2\bigr).
    \label{app:eq:energy_rate_majorant}
\end{equation}
Define
\begin{equation}
    \widehat\lambda
    :=\min\left\{
        \frac d{4P},\frac{\tau t}{2Q_0}
    \right\}.
    \label{app:eq:lambda_sharp}
\end{equation}
The dissipation in~\eqref{eq:closed_descent} is at least
$\widehat\lambda$ times the right-hand side of
\eqref{app:eq:energy_rate_majorant}, hence at least
$\widehat\lambda\mathcal H_k$.
Moreover, $P>A\ge d/2$, so $0<\widehat\lambda<1/2$.
It follows that
$\mathcal H_{k+1}\le(1-\widehat\lambda)\mathcal H_k$.

To recover the simpler coefficient in the theorem, set
$K_\mu:=7+8/(\gamma\mu)$.
The bounds $A\le2$, $1-t\le1$, $\beta^2\le1$, and $b^2\le4$
give
\[
    P\le4+\frac2{\gamma\mu}\le\frac{K_\mu}{2},
    ~~
    Q_0\le7+\frac8{\gamma\mu}=K_\mu.
\]
The second inequality for $P$ uses
$\gamma\mu<\gamma L<2b<4$.
Since $\tau<4$ and $t=d/(2b)$,
\[
    \frac d{4P}\ge\frac d{2K_\mu}
       \ge\frac{d\tau}{8K_\mu},
    ~~
    \frac{\tau t}{2Q_0}
       =\frac{d\tau}{4bQ_0}
       \ge\frac{d\tau}{8K_\mu}.
\]
Thus $\widehat\lambda\ge\lambda$, where
$\lambda:=d\tau/(8K_\mu)\in(0,1)$, and
\[
    \mathcal H_k
    \le(1-\widehat\lambda)^k\mathcal H_0
    \le(1-\lambda)^k\mathcal H_0.
\]
Combining this inequality with~\eqref{app:eq:position_energy}
gives the stated position bound. Equation~\eqref{eq:next_velocity}
also yields
\[
    \|v_{k+1}\|^2
    \le\frac{4\gamma N}{D}(1-\lambda)^k\mathcal H_0.
\]
Both position error and velocity therefore converge to zero,
completing the proof of Theorem~\ref{thm:invariant}.
The numerical energy envelopes use
$\widehat\lambda$ from~\eqref{app:eq:lambda_sharp}.

\subsection{Attraction radius}

For $e:=x-x_\star$, smoothness gives
$f(x)-f_\star\le L\|e\|^2/2$ and
$\|\nabla f(x)\|\le L\|e\|$.
Since $A\le2$ and $1-t\le1$,
\[
\begin{aligned}
    \mathcal H(x,v)
    &\le\frac L2\|e\|^2+\frac A{4\gamma}\|v\|^2
         +\frac L2\|e\|\|v\|\\
    &\le\frac{3L}{4}\|e\|^2
       +\left(\frac1{2\gamma}+\frac L4\right)\|v\|^2\\
    &\le C_H\|(e,v)\|_2^2,
\end{aligned}
\]
where $C_H:=\max\{3L/4,\,1/(2\gamma)+L/4\}$.
The closed ball of radius
$R_\nu\sqrt{D/(4\gamma NC_H)}$ around $(x_\star,0)$
is therefore contained in $\{\mathcal H\le\mathcal B_\nu\}$.
Theorem~\ref{thm:invariant} applies to every initial state in this
ball and every objective in the class, proving
\eqref{eq:radius_lower_bound}.

For Corollary~\ref{cor:radius_lower}, fix
$\mu,L,\gamma,\nu$ with $\gamma L<4$ and set
$\Delta:=4-\gamma L>0$.
As $d\downarrow0$, we have
$t/d\to1/4$, $\tau\to\Delta$, $D\to\Delta$, and $N\to2$.
The lower bound divided by $(d/M_\nu)^{1/\nu}$ consequently tends to
\[
    \left(\frac{(\nu+2)\Delta}{4\gamma}\right)^{1/\nu}
    \sqrt{\frac{\Delta}{8\gamma C_H}}>0.
\]
Take $c$ to be half of this limit and choose $d_0>0$ sufficiently
small that the ratio is at least $c$ and $\tau>0$ whenever
$0<d<d_0$.
These choices depend only on $\mu,L,\gamma,\nu$ and establish the
claim simultaneously for every $n\ge1$ and $M_\nu>0$.
The same lower bound applies after restriction to $C^\infty$
objectives.

\subsection{Spatial scaling}

Fix $n,\mu,L,\nu,\gamma,\beta$ and write
$\mathcal F_M:=\mathcal F_{\mu,L,M}^{n,\nu}$ for $M>0$.
Let $R:=M^{-1/\nu}$. For $f\in\mathcal F_1$, define
$f_R(x):=R^2f(x/R)$.
Then
$\nabla f_R(x)=R\nabla f(x/R)$ and
$\nabla^2f_R(x)=\nabla^2f(x/R)$.
The curvature bounds are preserved, and
\[
    \|\nabla^2f_R(x)-\nabla^2f_R(y)\|_{\mathrm{op}}
    \le R^{-\nu}\|x-y\|^\nu=M\|x-y\|^\nu.
\]
Conversely, if $F\in\mathcal F_M$, the function
$f(y):=R^{-2}F(Ry)$ has the same curvature bounds and Hessian
H\"older constant at most $MR^\nu=1$.
These transformations are inverses, so they give a bijection
between $\mathcal F_1$ and $\mathcal F_M$, including their
$C^\infty$ subclasses.

If $(x_k,v_k)$ is a heavy-ball trajectory for $f$, then
$(Rx_k,Rv_k)$ is the trajectory for $f_R$ from $(Rx_0,Rv_0)$
with the same parameters. Indeed,
\[
    \beta Rv_k-\gamma\nabla f_R(Rx_k)
    =R\bigl(\beta v_k-\gamma\nabla f(x_k)\bigr),
\]
and the position update scales in the same way.
The minimizer becomes $Rx_\star$ and every state distance from
equilibrium is multiplied by $R$.
Convergence of either trajectory is equivalent to convergence
of the other.

Define the set of admissible radii by
$\mathscr R(M):=\{r\ge0:\eqref{eq:uniform_guarantee}
\text{ holds for }\mathcal F_M\}$.
The class bijection and trajectory scaling show exactly that
$\mathscr R(M)=R\mathscr R(1)$.
Taking suprema proves~\eqref{eq:radius_scaling}, also when the
radius is infinite.

\section{Startup and Local Rates}
\label{app:asymptotic}

We prove the finite warm-up guarantee in Corollary~\ref{cor:startup}
and the local rate bound~\eqref{eq:startup_rate}. The target
heavy-ball parameters remain fixed during warm-up; they determine
the energy budget used by the stopping test.

\subsection{Finite warm-up}

\begin{proposition}[Warm-up length]
\label{app:prop:warmup}
Let $f\in\mathcal F_{\mu,L,M_\nu}^{n,\nu}$, and fix target parameters
$\gamma>0$ and $0\le\beta<1$ satisfying $\gamma L<2(1+\beta)$.
Let $\mathcal B_\nu$ be the positive budget in
\eqref{eq:energy_budget}. Starting from any $y_0\in\mathbb R^n$,
run $y_{j+1}=y_j-L^{-1}\nabla f(y_j)$ and define
\[
    J:=\min\bigl\{j\ge0:
        \|\nabla f(y_j)\|^2\le2\mu\mathcal B_\nu\bigr\}.
\]
Then $J<\infty$. If the initial state satisfies the test, $J=0$.
Otherwise, with $F_0:=f(y_0)-f_\star$,
$\rho_G:=1-\mu/L$, and $\lceil u\rceil_+:=\max\{0,\lceil u\rceil\}$,
\begin{equation}
    \begin{aligned}
        J
        \le
        \left\lceil
        \frac{\log\!\bigl(LF_0/(\mu\mathcal B_\nu)\bigr)}
             {-\log\rho_G}
        \right\rceil_+
        \le
        \left\lceil
        \frac{\log\!\bigl(
            L\|\nabla f(y_0)\|^2/(2\mu^2\mathcal B_\nu)
        \bigr)}
             {-\log\rho_G}
        \right\rceil_+.
    \end{aligned}
    \label{app:eq:warmup_length}
\end{equation}
Heavy-ball initialized at $(y_J,0)$ satisfies all conclusions of
Theorem~\ref{thm:invariant}.
\end{proposition}

\begin{proof}
Write $F_j:=f(y_j)-f_\star$. Strong convexity and smoothness give
\[
    2\mu F_j\le\|\nabla f(y_j)\|^2\le2LF_j.
\]
For the first inequality, minimize the strong-convexity lower bound
\[
    f(z)\ge f(y_j)+\langle\nabla f(y_j),z-y_j\rangle
        +\frac\mu2\|z-y_j\|^2
\]
over $z$. For the second, apply the descent lemma at
$y_j-L^{-1}\nabla f(y_j)$ and bound the resulting function value
below by $f_\star$.
The same descent lemma yields
\[
    F_{j+1}
    \le F_j-\frac{1}{2L}\|\nabla f(y_j)\|^2
    \le\rho_G F_j.
\]
It follows that
$\|\nabla f(y_j)\|^2\le2L\rho_G^jF_0$.
If $\nabla f(y_0)=0$, then $y_0=x_\star$ and $J=0$.
More generally, any initial state satisfying the gradient test has
$J=0$. Otherwise $F_0>0$, and the first bound in
\eqref{app:eq:warmup_length} follows by solving
$L\rho_G^jF_0\le\mu\mathcal B_\nu$ for an integer $j\ge0$.
The inequality $F_0\le\|\nabla f(y_0)\|^2/(2\mu)$ gives the
second bound, which uses only the initial gradient and known constants.

At the stopping point, the velocity reset gives
\[
    \mathcal H(y_J,0)
    =f(y_J)-f_\star
    \le\frac{\|\nabla f(y_J)\|^2}{2\mu}
    \le\mathcal B_\nu.
\]
Theorem~\ref{thm:invariant} now applies to the heavy-ball phase.
\end{proof}

\subsection{Local root rate}

\begin{proposition}[Root rate]
\label{app:prop:rootrate}
Suppose $f$ is $C^2$ near $x_\star$,
$\nabla f(x_\star)=0$, and
$\mu I\preceq\nabla^2f(x_\star)\preceq LI$ with $0<\mu<L$.
Let $\gamma>0$ and $0\le\beta<1$ satisfy
$\gamma L<2(1+\beta)$. Define
\[
    Q_\star:=\nabla^2f(x_\star),~~
    J_\star:=
    \begin{pmatrix}
        I-\gamma Q_\star&\beta I\\
        -\gamma Q_\star&\beta I
    \end{pmatrix}.
\]
Every heavy-ball trajectory converging to $(x_\star,0)$ satisfies
\[
    \limsup_{k\to\infty}
    \|(x_k-x_\star,v_k)\|_2^{1/k}
    \le\rho(J_\star)<1.
\]
For the quadratic-optimal parameters
$\gamma_\star:=4/(\sqrt L+\sqrt\mu)^2$ and
$\beta_\star:=q^2$, where
$q:=(\sqrt L-\sqrt\mu)/(\sqrt L+\sqrt\mu)$,
we have $\rho(J_\star)=q$.
\end{proposition}

\begin{proof}
For $z:=(e,v)$ with $e:=x-x_\star$, define the state map
\[
    \Phi(e,v):=
    \bigl(e+\beta v-\gamma\nabla f(x_\star+e),\,
          \beta v-\gamma\nabla f(x_\star+e)\bigr).
\]
It is continuously differentiable near zero, with
$\Phi(0)=0$ and $D\Phi(0)=J_\star$.
Orthogonally diagonalizing the symmetric matrix $Q_\star$ reduces
$J_\star$ to the blocks
\[
    J_\ell:=
    \begin{pmatrix}
        1-\gamma\ell&\beta\\
        -\gamma\ell&\beta
    \end{pmatrix},
    ~~ \ell\in\operatorname{spec}(Q_\star).
\]
Writing $b:=1+\beta$, their characteristic polynomials are
$p_\ell(r):=r^2-(b-\gamma\ell)r+\beta$.
Here $p_\ell(1)=\gamma\ell>0$ and
$p_\ell(-1)=2b-\gamma\ell>0$, while the product of the roots
is $\beta\in[0,1)$. If the roots are nonreal, their modulus is
$\sqrt\beta<1$. If they are real, a root at least one would force
the other below one and thus $p_\ell(1)\le0$; a root at most minus
one would similarly give $p_\ell(-1)\le0$.
Thus $\rho:=\rho(J_\star)<1$.

Fix $\rho<a<r<1$. The Jordan normal form implies
$\sup_{j\ge0}a^{-j}\|J_\star^j\|_2<\infty$: each block
contributes a polynomial in $j$ times a geometric sequence whose
ratio to $a^j$ decays. Consequently,
\[
    \|z\|_a:=\sup_{j\ge0}a^{-j}\|J_\star^jz\|_2
\]
defines a norm equivalent to $\|\cdot\|_2$, and
$\|J_\star z\|_a\le a\|z\|_a$.
Differentiability at zero gives
$\Phi(z)=J_\star z+o(\|z\|_a)$.
There is therefore a radius $\delta>0$ such that
\[
    \|\Phi(z)\|_a\le r\|z\|_a
    ~~\text{whenever }\|z\|_a\le\delta.
\]
The assumed convergence places some iterate $z_K$ in this ball.
The ball is invariant, and
$\|z_k\|_a\le r^{k-K}\|z_K\|_a$ for all $k\ge K$.
Norm equivalence then gives
$\limsup_k\|z_k\|_2^{1/k}\le r$.
Letting $r\downarrow\rho$ proves the root-rate bound.

At the quadratic-optimal parameters,
$b-\gamma_\star\mu=2q$ and
$b-\gamma_\star L=-2q$.
Each block therefore has polynomial
$r^2-c_\ell r+q^2$ with
$c_\ell:=b-\gamma_\star\ell\in[-2q,2q]$.
For $|c_\ell|<2q$, both roots have modulus $q$.
At the endpoints, the repeated root is $q$ or $-q$.
Thus $\rho(J_\star)=q$ in every case. The norm argument above
includes repeated roots and their possible polynomial transients.
\end{proof}

The quadratic-optimal parameters satisfy
$2(1+\beta_\star)-\gamma_\star L
=4\mu/(\sqrt L+\sqrt\mu)^2>0$.
Proposition~\ref{app:prop:warmup} and
Theorem~\ref{thm:invariant} ensure convergence after the startup test,
so Proposition~\ref{app:prop:rootrate} proves
\eqref{eq:startup_rate}. If the heavy-ball phase begins after $J$
gradient steps, its decay bound holds with index $k-J$.
Since $(k-J)/k\to1$, the finite warm-up leaves the asymptotic root
factor unchanged.

\section{Smooth Cycles}
\label{app:holder}

We prove Theorem~\ref{thm:sharp_radius}. Throughout this appendix,
$0<\mu<L$, $\gamma>0$ with $\gamma L<4$, and $0<\nu\le1$ are fixed.
Recall $d:=1-\beta$ and $b:=2-d$, and set $\Delta:=4-\gamma L>0$.
We first construct a scalar objective with a locally attracting cycle
at distance $O((d/M_\nu)^{1/\nu})$ from equilibrium.
We then combine this construction with the uniform energy bound.

\subsection{Nodes and interpolation}

The map $\theta\mapsto2(1-\cos\theta)$ takes $(0,\pi)$ continuously
and strictly increasingly onto $(0,4)$. The density of rational
multiples of $2\pi$ therefore gives coprime positive integers $m,K$
such that $\theta:=2\pi m/K\in(0,\pi)$ satisfies
\eqref{eq:resonance}. In particular, $K\ge3$.
Fix this choice and define
$a:=1-\cos\theta$, $s:=\sin\theta>0$, $\phi:=\pi/(2K)$,
$u_j:=j\theta-\phi$, and $z_j:=\cos u_j$ for $0\le j<K$.
Indices on this sequence are interpreted modulo $K$.

The nodes $z_j$ are pairwise distinct. Indeed, $\cos u_i=\cos u_j$
implies either $u_i-u_j\in2\pi\mathbb Z$ or
$u_i+u_j\in2\pi\mathbb Z$. The first alternative gives
$m(i-j)\in K\mathbb Z$, and coprimality yields $i=j$.
The second would give $2m(i+j)-1\in2K\mathbb Z$, which is
impossible by parity. The prescribed values $\sin u_j$ are not all
equal, since one sine value corresponds to at most two distinct
cosine values on the unit circle.

Define the smooth cutoff
\[
    \eta(q):=
    \begin{cases}
        e^{-1/q},&q>0,\\
        0,&q\le0,
    \end{cases}
    ~~
    \sigma(q):=
    \frac{\eta(q-1/3)}{\eta(q-1/3)+\eta(2/3-q)}.
\]
The denominator is positive everywhere. Thus $\sigma$ is smooth,
takes values in $[0,1]$, equals zero for $q\le1/3$, and equals one
for $q\ge2/3$.

Let $\zeta_1<\cdots<\zeta_K$ be the sorted nodes, and let $y_i$
denote the prescribed sine value at $\zeta_i$. On each interval
$[\zeta_i,\zeta_{i+1}]$, define
\[
    H(z):=y_i+(y_{i+1}-y_i)
    \sigma\!\left(\frac{z-\zeta_i}{\zeta_{i+1}-\zeta_i}\right).
\]
Extend $H$ by $y_1$ to the left and by $y_K$ to the right.
The pieces agree on constant neighborhoods of the nodes, so
$H\in C^\infty(\mathbb R)$, $H(z_j)=\sin u_j$, and
$\|H\|_\infty\le1$. Moreover, $H'$ has compact support.

Set $K_1:=\|H'\|_\infty$ and $K_2:=\|H''\|_\infty$.
Both are finite and positive: $H$ is nonconstant, and a nonzero
compactly supported derivative cannot be constant. For any $x,y$,
\[
    |H'(x)-H'(y)|
    \le\min\{K_2|x-y|,2K_1\}
    \le K_2^\nu(2K_1)^{1-\nu}|x-y|^\nu.
\]
The last inequality follows from
$\min\{p,q\}\le p^\nu q^{1-\nu}$ for $p,q>0$.
Consequently, $K_{H,\nu}:=[H']_\nu$ is finite.
It is also positive, since a zero seminorm would force $H'$ to be
constant and hence zero, contradicting the interpolation data.
The angle and profile depend only on $\mu,L,\gamma$, while
$K_{H,\nu}$ may also depend on $\nu$.

\subsection{The scalar objective}

Define
\[
    \delta:=\min\{2a-\gamma\mu,\gamma L-2a\}>0,
    ~~
    d_0:=\min\left\{
        \frac12,\frac{\Delta}{4},
        \frac{\delta}{2(a+sK_1)}
    \right\}.
\]
For $0<d<d_0$ and $M_\nu>0$, use $R_c$ from
\eqref{eq:cycle_construction} and define the objective explicitly by
\[
    f(x):=\frac{ba}{2\gamma}x^2
    +\frac{dR_c^2s}{\gamma}\int_0^{x/R_c}H(q)\,dq.
\]
Its derivative is exactly the prescribed expression in
\eqref{eq:cycle_construction}. The function is $C^\infty$, and
\[
    \gamma f''(x)=(2-d)a+dsH'(x/R_c),
    ~~
    |\gamma f''(x)-2a|
    \le d(a+sK_1)<\delta/2.
\]
Thus $\mu<f''(x)<L$ globally. Changing variables in the supremum
that defines the H\"older seminorm gives
\[
    [f'']_\nu
    =\frac{ds}{\gamma R_c^\nu}[H']_\nu
    =M_\nu.
\]
Hence $f\in\mathcal F_{\mu,L,M_\nu}^{1,\nu}\cap C^\infty$.
The choice of $d_0$ also ensures
$\tau=\Delta-2d>\Delta/2>0$.

Set $x_j:=R_c\cos u_j$ and $v_j:=x_j-x_{j-1}$.
The trigonometric addition formulas give
\[
    \begin{aligned}
        bx_j-x_{j+1}-\beta x_{j-1}
        &=R_c\bigl(ba\cos u_j+ds\sin u_j\bigr)\\
        &=\gamma f'(x_j).
    \end{aligned}
\]
It follows that $\beta v_j-\gamma f'(x_j)=v_{j+1}$ and
$x_j+v_{j+1}=x_{j+1}$. The resulting heavy-ball trajectory is
periodic, with least period $K$ because its position nodes are
distinct.

\subsection{Attraction and radius}

Let $\Phi(x,v):=(x+\beta v-\gamma f'(x),\,
\beta v-\gamma f'(x))$ denote the scalar heavy-ball map.
Since $H$ is constant near every $z_j$, the curvature at each
cycle position is $ba/\gamma$. The derivative of $\Phi$ at every
cycle phase is therefore
\[
    J:=\begin{pmatrix}
        1-ba&\beta\\
        -ba&\beta
    \end{pmatrix},
    ~~
    p(\zeta):=\det(\zeta I-J)
    =\zeta^2-b\cos\theta\,\zeta+\beta.
\]
We have $p(1)=b(1-\cos\theta)>0$,
$p(-1)=b(1+\cos\theta)>0$, and $0\le\beta<1$.
These conditions imply that both roots lie strictly inside the
unit disk. If the roots are nonreal, their common modulus is
$\sqrt\beta<1$. If they are real, a root at or beyond $1$ would,
by $p(1)>0$, require both roots to exceed $1$, contradicting their
product $\beta<1$. The same argument at $-1$ excludes a root at
or below $-1$.

For completeness, this linear stability gives local attraction as
follows. Every cycle phase $z$ is a fixed point of $\Phi^K$, and
$D(\Phi^K)(z)=J^K$. Put $B:=J^K$. Since $\rho(B)<1$, the series
$P:=\sum_{\ell=0}^{\infty}(B^\ell)^\top B^\ell$ converges to a
positive-definite matrix; convergence follows, for example, by
bounding powers of each Jordan block. It satisfies
$B^\top PB=P-I$. In the norm
$\|w\|_P:=(w^\top Pw)^{1/2}$, the induced norm of $B$ is strictly
less than one. By continuity of $D(\Phi^K)$, its induced norm is
bounded by some $q<1$ throughout a sufficiently small convex
neighborhood of $z$. The mean-value inequality then gives
$\|\Phi^K(y)-z\|_P\le q\|y-z\|_P$ on a small ball centered at
$z$. This ball is invariant under $\Phi^K$, and its iterates converge
to $z$. Continuity of $\Phi,\ldots,\Phi^{K-1}$ gives convergence of
the intermediate iterates to the corresponding cycle phases.
Thus the cycle attracts an open neighborhood in phase space.

Strong convexity implies
$\mu|x_\star|\le|f'(0)|\le dR_cs/\gamma$.
Since $|x_j|\le R_c$ and $|v_j|\le2R_c$, every cycle phase obeys
\[
    \begin{aligned}
        \|(x_j-x_\star,v_j)\|_2
        &\le R_c\sqrt{\left(1+\frac{ds}{\gamma\mu}\right)^2+4}\\
        &\le C\left(\frac{d}{M_\nu}\right)^{1/\nu},
    \end{aligned}
\]
where the constant
\[
    C:=\left(\frac{sK_{H,\nu}}{\gamma}\right)^{1/\nu}
    \sqrt{\left(1+\frac{d_0s}{\gamma\mu}\right)^2+4}
\]
is independent of $M_\nu$ and the dimension.
Every such state is nonconvergent. If its distance from equilibrium
is $R$, the uniform guarantee fails for every initialization radius
at least $R$. Taking the supremum in~\eqref{eq:uniform_radius}
therefore gives $r_{\mathrm{univ}}\le R$ and proves the scalar
upper bound.

\subsection{Dimension and matching bounds}

For $n>1$, define
$F(x):=f(x_1)+(\mu/2)\sum_{\ell=2}^n x_\ell^2$.
Its Hessian is diagonal with entries $f''(x_1),\mu,\ldots,\mu$,
so the curvature bounds are unchanged. Moreover,
\[
    \begin{aligned}
        \|\nabla^2F(x)-\nabla^2F(y)\|_{\mathrm{op}}
        &=|f''(x_1)-f''(y_1)|\\
        &\le M_\nu|x_1-y_1|^\nu
        \le M_\nu\|x-y\|^\nu.
    \end{aligned}
\]
Zero initial positions and velocities in the additional coordinates
preserve the scalar cycle and its distance from equilibrium.
Each added coordinate has characteristic polynomial
$p_\mu(\zeta):=\zeta^2-(b-\gamma\mu)\zeta+\beta$.
Here $p_\mu(1)=\gamma\mu>0$ and
$p_\mu(-1)=2b-\gamma\mu>0$, the latter following from $\tau>0$.
The preceding root argument gives strict stability. After permuting
coordinates, the derivative of the period map is block diagonal
with stable blocks. The same positive-definite-matrix argument
therefore proves local attraction in $\mathbb R^{2n}$.

We can use the same damping threshold $d_0$ for the lower bound.
For $d<d_0\le1/2$, we have
$t=d/(2b)\ge d/4$ and $t\le1/6$.
Also $\tau\ge\Delta/2$,
$D=\tau(1-t)^2+td\ge\Delta/4$, and
$N=2\beta^2+\gamma Ld/2\le3$.
Substitution into~\eqref{eq:radius_lower_bound} gives
\[
    r_{\mathrm{univ}}
    \ge c\left(\frac{d}{M_\nu}\right)^{1/\nu},
    ~~
    c:=\left(\frac{(\nu+2)\Delta}{8\gamma}\right)^{1/\nu}
    \sqrt{\frac{\Delta}{48\gamma C_H}}.
\]
Thus $c,C,d_0$ depend only on the fixed parameters.

Let $r_{\mathrm{univ}}^\infty$ denote the uniform radius restricted
to $C^\infty$ objectives. The lower guarantee holds for the full
class, while the upper-bound witness is smooth. Consequently,
\[
    c\left(\frac{d}{M_\nu}\right)^{1/\nu}
    \le r_{\mathrm{univ}}
    \le r_{\mathrm{univ}}^\infty
    \le C\left(\frac{d}{M_\nu}\right)^{1/\nu}.
\]
Together with local attraction and the dimension-independent
construction, this proves Theorem~\ref{thm:sharp_radius}.

\section{Three-Cycle Constants}
\label{app:sharp}

We prove Corollary~\ref{cor:explicit_radius} for
$\mu=\gamma=1$, $L=7/2$, and $0<d\le1/10$.
The construction uses one smooth interpolation profile independent
of $d$ and $M$.

\subsection{A smooth flat profile}

Define $S(u):=6u^5-15u^4+10u^3$ for $u\in[0,1]$.
Its derivative is $p(u):=S'(u)=30u^2(1-u)^2$, and
\[
    \|p\|_\infty=\frac{15}{8},
    ~~
    p'(u)=60u(1-u)(1-2u),
    ~~
    \|p'\|_\infty=\frac{10}{\sqrt3}.
\]
The maxima occur at $u=1/2$ for $p$ and at
$u=(3\pm\sqrt3)/6$ for $|p'|$.
Both $p$ and $p'$ vanish at the endpoints.

Using $\eta$ from Appendix~\ref{app:holder}, define
$\chi(q):=\eta(q-1)/(\eta(q-1)+\eta(2-q))$ and
$C_\chi:=\|\chi'\|_\infty$.
The function $\chi$ is smooth, takes values in $[0,1]$, equals zero
for $q\le1$, and equals one for $q\ge2$. Its derivative is
compactly supported, so $C_\chi<\infty$.
For $0<\epsilon<1/8$, define on $[0,1]$
\[
    \begin{aligned}
        \psi_\epsilon(u)
        :=\chi(u/\epsilon)\chi((1-u)/\epsilon)p(u),~~Z_\epsilon:=\int_0^1\psi_\epsilon(u)\,du,~~
        S_\epsilon(u):=Z_\epsilon^{-1}
        \int_0^u\psi_\epsilon(q)\,dq.
    \end{aligned}
\]
Here $Z_\epsilon>0$, because $\psi_\epsilon=p>0$ on
$(2\epsilon,1-2\epsilon)$.
Extend $S_\epsilon$ by zero for $u<0$ and by one for $u>1$.
The function $\psi_\epsilon$ vanishes near both endpoints, so
$S_\epsilon$ is $C^\infty$ and constant on their neighborhoods.

We now bound its first two derivatives explicitly.
The difference $\psi_\epsilon-p$ is supported in
$[0,2\epsilon]\cup[1-2\epsilon,1]$.
On these strips, $|p|\le120\epsilon^2$ and
$|p'|\le120\epsilon$. Only one cutoff can vary on each strip;
the other equals one. Therefore,
\[
    \begin{aligned}
        0\le1-Z_\epsilon\le480\epsilon^3,~~
        \|\psi_\epsilon\|_\infty\le15/8,~~
        \|\psi_\epsilon'-p'\|_\infty
        &\le120(1+C_\chi)\epsilon.
    \end{aligned}
\]
Choose and fix $\epsilon>0$ sufficiently small that
\[
    480\epsilon^3\le\frac1{16},
    ~~
    120(1+C_\chi)\epsilon+2880\epsilon^3
    \le6-\frac{10}{\sqrt3}.
\]
Such a choice exists because the right-hand side of the second
inequality is positive. The preceding estimates yield
\[
    \|S_\epsilon'\|_\infty
    \le\frac{15/8}{1-480\epsilon^3}\le2,
    ~~
    \|S_\epsilon''\|_\infty
    \le\frac{10/\sqrt3+120(1+C_\chi)\epsilon}
             {1-480\epsilon^3}\le6.
\]
Finally, define $H(u):=-1+(3/2)S_\epsilon(|u|)$.
It is smooth at zero because it is constant nearby, and smooth
elsewhere by construction. It is constant near $-1,0,1$ and
satisfies
\[
    H(0)=-1,
    ~~ H(-1)=H(1)=\frac12,
    ~~ \|H'\|_\infty\le3,
    ~~ \|H''\|_\infty\le9.
\]
This proves the profile properties used in
Section~\ref{sec:explicit_cycle}.

\subsection{The periodic obstruction}

Set $R_c:=9d/M$ and define
\[
    f(x):=\frac12\left(3-\frac32d\right)x^2
    +dR_c^2\int_0^{x/R_c}H(q)\,dq.
\]
Then $f\in C^\infty$ and
\[
    \begin{aligned}
        f'(x)=\left(3-\frac32d\right)x+dR_cH(x/R_c),~
        f''(x)=3-\frac32d+dH'(x/R_c),~
        f'''(x)=\frac{d}{R_c}H''(x/R_c).
    \end{aligned}
\]
Hence
$3-9d/2\le f''(x)\le3+3d/2$ and
$\|f'''\|_\infty\le M$.
For $d\le1/10$, these curvature bounds lie within $[1,7/2]$,
and the mean-value theorem gives $\operatorname{Lip}(f'')\le M$.
Thus $f\in\mathcal F_{1,7/2,M}^{1,1}$.

At the three interpolation nodes,
\[
    f'(-R_c)=(-3+2d)R_c,
    ~~ f'(0)=-dR_c,
    ~~ f'(R_c)=(3-d)R_c.
\]
With $\beta=1-d$ and $\gamma=1$, the corresponding phase states
follow the exact cycle
\[
    (-R_c,-2R_c)\longmapsto(0,R_c)
    \longmapsto(R_c,R_c)\longmapsto(-R_c,-2R_c).
\]
The Hessian at each node is $3-3d/2=3b/2$.
The common phase derivative has characteristic polynomial
$p_3(\zeta):=\zeta^2+(b/2)\zeta+\beta$.
It satisfies $p_3(1)=3b/2>0$, $p_3(-1)=b/2>0$, and
$0\le\beta<1$. The root and period-map arguments in
Appendix~\ref{app:holder} therefore prove local attraction.

Since $f'(0)<0<f'(R_c)$, the minimizer satisfies
$0<x_\star<R_c$. Integrating $f''\ge1$ from $0$ to $x_\star$
gives $x_\star\le dR_c$. The phase $(0,R_c)$ is consequently at
distance
\[
    \|(-x_\star,R_c)\|_2
    \le R_c\sqrt{1+d^2}
    \le9\sqrt{1.01}\,\frac{d}{M}
    <10\frac{d}{M}
\]
from equilibrium. This proves $r_{\mathrm{univ}}\le10d/M$.
For $n>1$, adding $(1/2)\sum_{\ell=2}^n x_\ell^2$ preserves the
function-class bounds, the cycle distance, and local attraction,
as proved in Appendix~\ref{app:holder}. The upper bound therefore
holds in every dimension and on the $C^\infty$ subclass.

\subsection{The uniform guarantee}

For the parameters of this corollary,
$b\in[19/10,2)$, $\tau=1/2-2d\ge3/10$, and
$t=d/(2b)\le1/38$. Thus
\[
    D=\tau(1-t)^2+td
    \ge\frac3{10}\left(\frac{37}{38}\right)^2
    >\frac7{25}.
\]
Furthermore,
$N=2(1-d)^2+(7/4)d=2-(9/4)d+2d^2\le2$
and $C_H=21/8$.
For $\nu=1$, the step budget in~\eqref{eq:energy_budget} satisfies
\[
    R_1=\frac{3\tau t}{M}
    =\frac{3\tau d}{2bM}
    \ge\frac{9d}{40M}.
\]
It follows from~\eqref{eq:radius_lower_bound} that
\[
    \begin{aligned}
        r_{\mathrm{univ}}
        \ge R_1\sqrt{\frac{D}{4NC_H}}
        \ge R_1\sqrt{\frac{D}{21}}
        >\frac{9d}{40M}\sqrt{\frac{7}{525}}
        =\sqrt{\frac{27}{25}}\,\frac{d}{40M}
        >\frac{d}{40M}.
    \end{aligned}
\]
More precisely, the ball of radius $d/(40M)$ is contained in the
invariant energy region for every admissible objective, so
Theorem~\ref{thm:invariant} applies to every initialization in
this closed ball.

The theorem's contraction parameter obeys
$\lambda=d\tau/[8(7+8/(\gamma\mu))]=d\tau/120\ge d/400$.
Consequently,
$\mathcal H_k\le(1-d/400)^k\mathcal H_0$ on the stated ball.
The lower guarantee applies to the full function class and hence
also to its smooth subclass. Together with the smooth periodic
obstruction, this proves Corollary~\ref{cor:explicit_radius}.

\section{Global Non-Acceleration}
\label{app:nonaccel}

We prove the parameter coverage and rate bounds in
Section~\ref{sec:global_obstruction}. Building on the projected-cycle construction of
\citet{fotopoulos2026}, we use a fixed phase and an elementary
period choice to obtain an attracting cycle whenever
$L/\mu\ge24$ and
$12(1-\beta)\le\gamma L<2(1+\beta)$.
Quadratic objectives cover the remaining parameters and yield
\eqref{eq:global_nonacceleration}, the global non-acceleration
phenomenon established by \citet{goujaud2025,fotopoulos2026}.

\subsection{Uniform factor}

Fix $0<\mu<L$, $M_\nu>0$, and $0<\nu\le1$, and recall
$\mathcal G:=\mathcal F_{\mu,L,M_\nu}^{1,\nu}$ and $\kappa:=L/\mu$.
For fixed parameters $\gamma>0$ and $0\le\beta<1$, define
\[
    Q_k(\gamma,\beta)
    :=
    \sup_{\substack{f\in\mathcal G\\z_0\ne0}}
    \frac{\|z_k\|_2}{\|z_0\|_2},
    ~~
    \rho_{\gamma,\beta}
    :=\limsup_{k\to\infty}Q_k(\gamma,\beta)^{1/k},
\]
where $z_k:=(x_k-x_\star,v_k)$ as in~\eqref{eq:global_factor}.
Then $\rho_\star=\inf_{\gamma>0,\,0\le\beta<1}\rho_{\gamma,\beta}$.
Since $|f'(x)|\le L|x-x_\star|$, the heavy-ball map satisfies
$\|z_{k+1}\|_2\le C_{\gamma,\beta,L}\|z_k\|_2$, uniformly over
$f\in\mathcal G$, for
$C_{\gamma,\beta,L}:=((1+\gamma L)^2+(\gamma L)^2+2\beta^2)^{1/2}$.
Thus every $Q_k$ is finite. If an admissible objective has a
nonconstant $K$-cycle, starting at a cycle phase gives
$Q_{mK}(\gamma,\beta)\ge1$ for every integer $m\ge1$, and therefore
$\rho_{\gamma,\beta}\ge1$.

\begin{theorem}[Global non-acceleration]
\label{app:nonaccel:theorem}
If $\kappa\ge24$, then
\begin{equation}
    \frac{\kappa-12}{\kappa+12}
    \le\rho_\star
    \le\frac{\kappa-1}{\kappa+1}.
    \label{app:eq:nonaccel_bounds}
\end{equation}
The same bounds hold when $\mathcal G$ is restricted to $C^\infty$
objectives. Moreover, every parameter pair satisfying
$12(1-\beta)\le\gamma L<2(1+\beta)$ admits a $C^\infty$ objective
in $\mathcal G$ with a locally attracting nonconstant periodic orbit.
Consequently, convergence for every objective and every initial state
requires $\gamma L<12(1-\beta)$.
\end{theorem}

Throughout the proof, $b:=1+\beta$, $d:=1-\beta$, and $h:=\gamma L$.

\subsection{Cycle secants}

For an integer $K\ge3$, set $\theta:=2\pi/K$,
$u_j:=j\theta-\theta/4$, and $x_j:=\cos u_j$ for $0\le j<K$.
Indices are interpreted modulo $K$. These positions are distinct:
equality of two cosines requires either $i-j\in K\mathbb Z$ or
$i+j-\tfrac12\in K\mathbb Z$. The second condition is impossible,
and the first implies $i=j$.

Prescribe the gradient values
\begin{equation}
    g_j
    :=\frac{b x_j-\beta x_{j-1}-x_{j+1}}{\gamma}.
    \label{app:eq:nonaccel_gradients}
\end{equation}
Any objective interpolating these values has the prescribed
$K$-cycle. With $a_K:=b(1-\cos\theta)$ and
$\eta_K:=d\sin\theta$, the cycle identity gives
$\gamma g_j=a_K\cos u_j+\eta_K\sin u_j$.
For $i\ne j$, the sum-to-product identities therefore yield
\begin{equation}
    \gamma\frac{g_i-g_j}{x_i-x_j}
    =a_K-\eta_K
    \cot\left(\frac{\pi(i+j-\tfrac12)}{K}\right).
    \label{app:eq:nonaccel_secants}
\end{equation}
In~\eqref{app:eq:nonaccel_secants}, the cotangent argument has
distance at least $\pi/(2K)$ from
$\pi\mathbb Z$. Since $\cot$ decreases on $(0,\pi/2]$, every
normalized secant lies in $[a_K-D_K,a_K+D_K]$, where
\[
    \begin{aligned}
        D_K
        :=d\sin\frac{2\pi}{K}\cot\frac{\pi}{2K}
        =4d\cos^2\frac{\pi}{2K}\cos\frac{\pi}{K}
        <4d.
    \end{aligned}
\]
This proves the secant estimate in
Section~\ref{sec:global_obstruction}.

Suppose now that $12d\le h<2b$. Choose the smallest integer
$K\ge3$ satisfying $a_K\le2h/3$. Such an integer exists because
$a_K\to0$ as $K\to\infty$. Since
$a_3=3b/2>2h/3$, the selected period satisfies $K\ge4$ and
$a_{K-1}>2h/3$. Concavity of sine on $[0,\pi]$, applied with
$\pi/K=((K-1)/K)(\pi/(K-1))$, gives
\[
    \frac{a_K}{a_{K-1}}
    =\left(
        \frac{\sin(\pi/K)}{\sin(\pi/(K-1))}
      \right)^2
    \ge\left(\frac{K-1}{K}\right)^2
    \ge\frac9{16}.
\]
Consequently, $3h/8<a_K\le2h/3$. Combining this with
$D_K<4d\le h/3$ and $\kappa\ge24$ yields
\[
    \begin{aligned}
        a_K-D_K
        &>\frac{3h}{8}-\frac h3
          =\frac h{24}\ge\frac h\kappa=\gamma\mu,\\
        a_K+D_K
        &<\frac{2h}{3}+\frac h3=h=\gamma L.
    \end{aligned}
\]
Thus every derivative secant prescribed by
\eqref{app:eq:nonaccel_gradients} lies strictly between $\mu$ and $L$.

\subsection{Smooth realization}

Sort the positions as $y_1<\cdots<y_K$ and let
$\widetilde g_i$ be the corresponding prescribed gradient values.
Define $\Delta_i:=y_{i+1}-y_i$ and
$s_i:=(\widetilde g_{i+1}-\widetilde g_i)/\Delta_i$ for
$1\le i<K$. The preceding bounds give $s_i\in(\mu,L)$.
Fix a common node curvature $c:=(\mu+L)/2$.

For each interval, choose a smooth cutoff
$\psi_i\in C_c^\infty((y_i,y_{i+1}))$ satisfying
$0\le\psi_i\le1$. Its integral
$J_i:=\int_{y_i}^{y_{i+1}}\psi_i(x)\,dx$ can be made arbitrarily
close to $\Delta_i$ by taking $\psi_i=1$ except near the endpoints.
Choose it close enough that
\[
    \left(\frac{\Delta_i}{J_i}-1\right)|s_i-c|
    <\min\{s_i-\mu,L-s_i\}
\]
whenever $s_i\ne c$. Define
\[
    w(x)
    :=c+(s_i-c)\frac{\Delta_i}{J_i}\psi_i(x),
    ~~ y_i\le x\le y_{i+1},
\]
and set $w(x):=c$ outside $[y_1,y_K]$.
On each interval, $w$ lies between $c$ and
$c+(s_i-c)\Delta_i/J_i$; both belong to $(\mu,L)$ by the choice
of $J_i$. The cutoffs vanish near their endpoints, so $w$ is globally
smooth and equals $c$ near every interpolation node. Also,
$\int_{y_i}^{y_{i+1}}w(x)\,dx=s_i\Delta_i$.

Define
\[
    g(x):=\widetilde g_1+\int_{y_1}^x w(s)\,ds,
    ~~
    f(x):=\int_{y_1}^x g(s)\,ds.
\]
The interval integral identities give $g(y_i)=\widetilde g_i$
successively for all $i$. Hence $f\in C^\infty(\mathbb R)$,
$\mu<f''<L$, and $f'(x_j)=g_j$. Equation
\eqref{app:eq:nonaccel_gradients} then gives the exact nonconstant
cycle with phases $(x_j,x_j-x_{j-1})$.

\paragraph{Local attraction.}
Near each node the gradient is affine with the same slope $c$.
Writing $e_k$ and $u_k$ for position and velocity errors relative
to a cycle trajectory, the error update in these neighborhoods is
\[
    \binom{e_{k+1}}{u_{k+1}}
    =J_c\binom{e_k}{u_k},
    ~~
    J_c:=
    \begin{pmatrix}
        1-\gamma c&\beta\\
        -\gamma c&\beta
    \end{pmatrix}.
\]
Its characteristic polynomial is
$p_c(\zeta):=\zeta^2-(b-\gamma c)\zeta+\beta$.
We have $p_c(1)=\gamma c>0$,
$p_c(-1)=2b-\gamma c>0$, and $0\le\beta<1$ because
$0<\gamma c<h<2b$. If the roots are nonreal, their common modulus
is $\sqrt\beta<1$. If they are real, their product is
$\beta\in[0,1)$; a root at least $1$ would force $p_c(1)\le0$,
and a root at most $-1$ would force $p_c(-1)\le0$.
Thus both roots lie strictly inside the unit disk.

Choose $r\in(\rho(J_c),1)$ and $C\ge1$ such that
$\|J_c^k\|_2\le Cr^k$ for every $k\ge0$.
Let $\varepsilon>0$ be small enough that every interval
$(x_j-\varepsilon,x_j+\varepsilon)$ lies inside an affine
neighborhood of its node. If the initial position--velocity error
has norm less than $\varepsilon/C$, induction gives
$\|(e_k,u_k)\|_2\le Cr^k\|(e_0,u_0)\|_2<\varepsilon$.
Every update therefore remains in the appropriate neighborhood,
and the error converges to zero. This proves local attraction of
the cycle in position--velocity space.

\paragraph{Hessian regularity.}
The derivative $w'$ has compact support, so
$B:=\|w'\|_\infty<\infty$. For all $x,y\in\mathbb R$,
\[
    |f''(x)-f''(y)|
    \le\min\{B|x-y|,L-\mu\}
    \le B^\nu(L-\mu)^{1-\nu}|x-y|^\nu.
\]
Thus $[f'']_\nu<\infty$. For $\ell>0$, define
$f_\ell(x):=\ell^2f(x/\ell)$. Then
$f_\ell''(x)=f''(x/\ell)$ and
$[f_\ell'']_\nu=\ell^{-\nu}[f'']_\nu$.
Taking
$\ell:=\max\{1,([f'']_\nu/M_\nu)^{1/\nu}\}$
gives $f_\ell\in\mathcal G\cap C^\infty$.
The change of state $(x,v)\mapsto(\ell x,\ell v)$ conjugates the
two heavy-ball maps with the same parameters. It therefore preserves
the periodic orbit and its local attraction. This proves the cycle
assertion of Theorem~\ref{app:nonaccel:theorem}.

\subsection{Rate bounds}

\begin{proof}[Proof of Theorem~\ref{app:nonaccel:theorem}]
Fix any $\gamma>0$ and $0\le\beta<1$.
We first prove the lower bound for this parameter pair.

If $h\ge2b$, choose $f(x):=Lx^2/2$.
Its iteration matrix has characteristic polynomial
$p_L(\zeta):=\zeta^2-(b-h)\zeta+\beta$.
Since $p_L(-1)=2b-h\le0$ and $p_L(\zeta)\to+\infty$ as
$\zeta\to-\infty$, it has a real root at or below $-1$.
Initializing along the corresponding eigenvector produces a
nondecaying trajectory, and $\rho_{\gamma,\beta}\ge1$.

If $12d\le h<2b$, the preceding construction gives a nonconstant
periodic orbit for a $C^\infty$ objective in $\mathcal G$.
Again, $\rho_{\gamma,\beta}\ge1$.

It remains to consider $h<2b$ and $h<12d$.
Choose $f(x):=\mu x^2/2$ and set $\alpha:=\gamma\mu=h/\kappa$.
Let $r$ be the largest modulus of the roots of
$\zeta^2-(b-\alpha)\zeta+\beta$.
For this quadratic, the supremum over initial states is exactly
the operator norm of the $k$th power of its position--velocity
iteration matrix. The spectral radius formula therefore implies
$\rho_{\gamma,\beta}\ge r$.
If $r\ge1$, the required lower bound follows immediately.
The case $r=0$ would imply $\beta=0$ and $\alpha=1$, whereas
$\alpha<12d/\kappa\le d/2$. We may consequently assume $0<r<1$.

The product of the roots gives $\beta\le r^2$, and their sum satisfies
$b-\alpha\le r+\beta/r$. Indeed, for two nonnegative real roots
this is equality; for two nonpositive real roots it follows from
the sign of their sum; for conjugate roots, $\beta=r^2$ and the
sum is at most $2r$. The same inequality holds when one root is
zero. Rearranging gives
\[
    \alpha\ge(1-r)\left(1-\frac\beta r\right).
\]
Moreover,
\[
    \frac{1-\beta/r}{1-\beta}-\frac1{1+r}
    =\frac{r^2-\beta}{r(1-\beta)(1+r)}\ge0.
\]
It follows that $\alpha/d\ge(1-r)/(1+r)$, and hence
\[
    \rho_{\gamma,\beta}
    \ge r
    \ge\frac{d-\alpha}{d+\alpha}
    >\frac{1-12/\kappa}{1+12/\kappa}
    =\frac{\kappa-12}{\kappa+12}.
\]
Taking the infimum over all parameter pairs proves the non-strict
lower bound in~\eqref{app:eq:nonaccel_bounds}.

For the upper bound, choose $\beta:=0$ and
$\gamma:=2/(L+\mu)$, and set
$r_{\mathrm{GD}}:=(L-\mu)/(L+\mu)\in(0,1)$.
For every $f\in\mathcal G$, integration of the Hessian along the
segment from $x_\star$ to $x$ gives
$f'(x)=a(x)(x-x_\star)$ for some $a(x)\in[\mu,L]$.
Therefore
\[
    |x_{k+1}-x_\star|
    =|1-\gamma a(x_k)|\,|x_k-x_\star|
    \le r_{\mathrm{GD}}|x_k-x_\star|.
\]
The initial velocity has no effect on these GD updates. For $k\ge1$,
$v_k=x_k-x_{k-1}$ consequently satisfies
$|v_k|\le(1+r_{\mathrm{GD}})r_{\mathrm{GD}}^{k-1}|x_0-x_\star|$.
Combining the position and velocity estimates gives the uniform bound
\[
    Q_k(\gamma,0)
    \le\sqrt{r_{\mathrm{GD}}^2+(1+r_{\mathrm{GD}})^2}\,
        r_{\mathrm{GD}}^{k-1}.
\]
Taking $k$th roots and the limit yields
$\rho_{\gamma,0}\le r_{\mathrm{GD}}$, proving the upper bound.
The curvature-$\mu$ quadratic gives the reverse inequality for this
GD pair, so its uniform factor equals $r_{\mathrm{GD}}$.

All lower-bound objectives are $C^\infty$ and have the prescribed
Hessian H\"older bound, while the GD estimate holds for the whole
class. The same argument thus proves the bounds on the smooth
subclass. Finally, the first two parameter cases each admit a
nonconvergent trajectory. Global convergence therefore requires
$h<2b$ and $h<12d$, proving the stated necessary condition.
\end{proof}

The bounds imply
$2/(\kappa+1)\le1-\rho_\star\le24/(\kappa+12)$.
Thus $1-\rho_\star=\Theta(\kappa^{-1})$, with constants independent
of $M_\nu$ and $\nu$.

\section{Necessary Cycle Size}
\label{app:cycle_size}

A work identity provides a lower bound on the size of every
nonconstant periodic orbit. This complements the constructed cycles
in Section~\ref{sec:cycles}.

\begin{proposition}[Necessary cycle scale]
\label{app:prop:cycle_size}
Let $f\in\mathcal F_{\mu,L,M_\nu}^{n,\nu}$, with $M_\nu>0$,
and let $\gamma>0$ and $0\le\beta<1$ satisfy
$\tau=2(1+\beta)-\gamma L>0$.
Every nonconstant $K$-periodic heavy-ball orbit satisfies
\[
    \operatorname{diam}\{x_0,\ldots,x_{K-1}\}
    \ge\max_i\|x_i-x_{i-1}\|
    \ge
    \left(\frac{2(\nu+2)d\tau}{\gamma bM_\nu}\right)^{1/\nu}.
\]
\end{proposition}

\begin{proof}
All indices are taken modulo $K$, and all sums run over one period.
Define $v_i:=x_i-x_{i-1}$, $g_i:=\nabla f(x_i)$,
$S:=\sum_i\|v_i\|^2$, and
$C:=\sum_i\langle v_i,v_{i+1}\rangle$.
Nonconstancy gives $S>0$, and the recurrence gives
$\gamma g_i=\beta v_i-v_{i+1}$.
The trapezoidal work over a period is
\begin{equation}
    \begin{aligned}
        W
        &:=\frac12\sum_i
            \langle g_i+g_{i+1},v_{i+1}\rangle\\
        &=\frac{1}{2\gamma}\sum_i
            \langle\beta v_i-dv_{i+1}-v_{i+2},v_{i+1}\rangle\\
        &=-\frac{d}{2\gamma}(S+C)
         =-\frac{d}{4\gamma}\sum_i\|v_i+v_{i+1}\|^2.
    \end{aligned}
    \label{app:eq:cycle_work}
\end{equation}
The potential increments sum to zero, so $W$ is also the sum of
the segmentwise trapezoidal errors.
For a segment from $x$ to $x+w$, write $\psi(u):=f(x+uw)$ and
\[
    T(x,w):=\tfrac12\bigl(\psi'(0)+\psi'(1)\bigr)
                 -\bigl(\psi(1)-\psi(0)\bigr).
\]
Integration by parts and reflection about $u=1/2$ give
\[
    \begin{aligned}
        T(x,w)
        &=\int_0^1(u-\tfrac12)\psi''(u)\,\mathrm du\\
        &=\int_{1/2}^1(u-\tfrac12)
            \langle w,[\nabla^2f(x+uw)
            -\nabla^2f(x+(1-u)w)]w\rangle\,\mathrm du.
    \end{aligned}
\]
Since
$\int_{1/2}^1(u-1/2)(2u-1)^\nu\,\mathrm du=1/[4(\nu+2)]$,
we obtain $|T(x,w)|\le M_\nu\|w\|^{2+\nu}/[4(\nu+2)]$.
Taking absolute values in~\eqref{app:eq:cycle_work} therefore yields
\begin{equation}
    \frac{d}{4\gamma}\sum_i\|v_i+v_{i+1}\|^2
    \le\frac{M_\nu}{4(\nu+2)}\sum_i\|v_i\|^{2+\nu}.
    \label{app:eq:cycle_work_bound}
\end{equation}

To control cancellation between consecutive velocities, use
$\langle g_i-g_{i-1},v_i\rangle\le L\|v_i\|^2$,
which follows by integrating the Hessian along the segment
$[x_{i-1},x_i]$. The recurrence also gives
$\gamma(g_i-g_{i-1})=bv_i-v_{i+1}-\beta v_{i-1}$.
Taking inner products with $v_i$ and summing gives
\[
    b(S-C)
    =\gamma\sum_i\langle g_i-g_{i-1},v_i\rangle
    \le\gamma LS.
\]
It follows that
\begin{equation}
    \sum_i\|v_i+v_{i+1}\|^2
    =2(S+C)\ge\frac{2\tau}{b}S.
    \label{app:eq:cycle_curvature}
\end{equation}
With $R:=\max_i\|v_i\|$, we have
$\sum_i\|v_i\|^{2+\nu}\le R^\nu S$.
Substituting~\eqref{app:eq:cycle_curvature} into
\eqref{app:eq:cycle_work_bound} and cancelling $S>0$ gives
$R^\nu\ge2(\nu+2)d\tau/(\gamma bM_\nu)$.
Every displacement joins two orbit nodes, so the diameter is at least $R$.
\end{proof}

For Lipschitz Hessians the lower bound becomes
$6d\tau/(\gamma bM_1)$. The proof applies to repeated nodes and
does not use attraction of the orbit. If the Hessian is constant,
all trapezoidal errors vanish, and
\eqref{app:eq:cycle_work}--\eqref{app:eq:cycle_curvature} rule out
nonconstant cycles under $d,\tau>0$.

\section{Cycles from Zero Velocity}
\label{app:zero_init}

The following example describes an interval of zero-velocity initial
states attracted to a periodic orbit at the quadratic-optimal
parameters. It clarifies the role of the gradient test in
Corollary~\ref{cor:startup}: the test selects a convergent initial
position for the velocity reset.

\begin{proposition}[Zero-velocity cycling]
\label{app:prop:zero_init}
There exists $f\in C^\infty(\mathbb R)$ with $1\le f''\le100$
and bounded third derivative such that heavy-ball with
$\gamma=4/121$ and $\beta=81/121$ has the locally attracting cycle
$-1\to0\to1\to-1$.
Every initialization $x_{-1}=x_0\in[37/50,19/25]$ converges to this
cycle. The curvature bounds $1$ and $100$ are attained on half-lines,
and the stated parameters are quadratic-optimal for these bounds.
\end{proposition}

\begin{proof}
Define $(t)_+:=\max\{t,0\}$ and
\[
    \begin{aligned}
        g_0(x):={}\frac{101}{2}x-\frac{81}{4}
        &+\frac{82}{7}\bigl[(x+15/16)_+-(x+1/16)_+\bigr]\\
        &+46\bigl[(x-1/16)_+-(x-15/16)_+\bigr]\\
        &+\frac{99}{2}(-x-2)_+
         +\frac{99}{2}(x-2)_+.
    \end{aligned}
\]
This continuous piecewise affine function has breakpoints
$-2,-15/16,-1/16,1/16,15/16,2$ and successive slopes
\[
    1,~\frac{101}{2},~\frac{871}{14},~
    \frac{101}{2},~\frac{193}{2},~
    \frac{101}{2},~100.
\]
Choose a nonnegative even mollifier
$\varphi\in C_c^\infty((-1,1))$ with integral one.
Set $\varepsilon:=1/100$,
$\varphi_\varepsilon(t):=\varepsilon^{-1}\varphi(t/\varepsilon)$,
$g:=g_0*\varphi_\varepsilon$, and $f(x):=\int_0^xg(t)\,\mathrm dt$.
The almost-everywhere derivative of $g_0$ lies in $[1,100]$,
so convolution gives $1\le g'\le100$.
Also $g''=g_0'*\varphi_\varepsilon'$ and
$\|g''\|_\infty\le100\|\varphi_\varepsilon'\|_1<\infty$.
Thus $f$ is smooth with the asserted curvature and regularity.
The extreme curvatures are attained because $f''=1$ on
$(-\infty,-2-\varepsilon]$ and $f''=100$ on
$[2+\varepsilon,\infty)$.

Convolution preserves $g_0$ at every point farther than
$\varepsilon$ from its breakpoints. Indeed, $g_0$ is affine on the
convolution interval, and evenness gives
$\int t\varphi_\varepsilon(t)\,\mathrm dt=0$.
In particular, with $r:=1/16-1/100=21/400$,
$g$ is affine with slope $c:=101/2$ on the open radius-$r$
neighborhood of each of $-1,0,1$. Its values at the nodes are
\[
    (g(-1),g(0),g(1))=(-283/4,-10,323/4).
\]
Substitution into
$x_{k+1}=bx_k-\beta x_{k-1}-\gamma g(x_k)$ verifies the cycle.
Since $b-\gamma c=0$, a nearby trajectory with phase errors
$e_k:=x_k-\bar x_k$, where $\bar x_k$ follows the exact cycle,
satisfies
\begin{equation}
    e_{k+1}=-\beta e_{k-1}
    \label{app:eq:zero_error}
\end{equation}
whenever its current position belongs to the corresponding affine
neighborhood. If two consecutive errors have magnitude less than
$r$, this identity keeps all later errors in those neighborhoods
and makes them tend to zero. This proves local attraction in the
two-position state. The transformation
$(x_k,x_{k-1})\mapsto(x_k,x_k-x_{k-1})$ is invertible, so attraction
also holds in position--velocity coordinates.

It remains to show that the full zero-velocity interval enters
these neighborhoods. For $z\in[37/50,19/25]$, first use $g_0$ and
initialize $x_{-1}=x_0=z$.
The required affine branches and their updates are
\[
\begin{array}{c|c|c}
    k&g_0(x_k)&x_{k+1}\\\hline
    0,2&
    \frac{193}{2}x_k-\frac{103}{8}&
    -\frac{184}{121}x_k-\frac{81}{121}x_{k-1}
        +\frac{103}{242}\\
    1,4,7&
    \frac{101}{2}x_k-\frac{81}{4}&
    -\frac{81}{121}x_{k-1}+\frac{81}{121}\\
    3,6&
    \frac{101}{2}x_k+\frac{121}{4}&
    -\frac{81}{121}x_{k-1}-1\\
    5&
    \frac{101}{2}x_k-10&
    -\frac{81}{121}x_{k-1}+\frac{40}{121}
\end{array}
\]
Successive substitution at the two endpoint values of $z$ gives
the following rational enclosures. The integer endpoints in the
last two columns are divided by $10^6$:
\[
\begin{array}{c|rr}
    k&10^6\underline x_k&10^6\overline x_k\\\hline
    0&740000&760000\\
    1&-1238843&-1195041\\
    2&160661&174050\\
    3&960935&1010617\\
    4&-1116513&-1107550\\
    5&-7108&26151\\
    6&1071996&1077996\\
    7&-1017506&-995242\\
    8&-52213&-48195
\end{array}
\]
To verify the enclosures, write $x_k=\alpha_k z+\eta_k$, starting
from $\alpha_{-1}=\alpha_0=1$ and
$\eta_{-1}=\eta_0=0$. Each update in the branch table determines
the next two rational coefficients. At each step, the endpoint
images lie in the same affine branch; therefore the recurrence
remains affine in $z$, and its extrema over the interval occur at
the endpoint initializations. For example, the last two updates give
\[
    x_7=\frac{238617009z-394687966}{214358881},
    ~~
    x_8=\frac{43046721(z-1)}{214358881}.
\]
This calculation keeps the dependence of successive states on the
same initialization $z$.

Every enclosure is farther than $\varepsilon=1/100$ from all
breakpoints of $g_0$. Thus $g=g_0$ on every visited interval,
and the same finite trajectory calculation holds exactly for the
smoothed objective. Taking the reference cycle phases
$\bar x_7=-1$ and $\bar x_8=0$, the table gives
\[
    |x_7+1|\le\frac{17506}{10^6}<\frac{21}{400},
    ~~
    |x_8|\le\frac{52213}{10^6}<\frac{21}{400}.
\]
Equation~\eqref{app:eq:zero_error} now proves convergence to the
cycle for every $z\in[37/50,19/25]$.
Finally, the quadratic-optimal formulas with $\mu=1$ and $L=100$
give $\gamma=4/121$ and $\beta=(9/11)^2=81/121$.
\end{proof}

The example can satisfy any prescribed positive Hessian H\"older
bound. For $0<\nu\le1$, define $M_{0,\nu}:=[f'']_\nu$.
The constructed Hessian is bounded and Lipschitz, so
$M_{0,\nu}<\infty$; explicitly,
$M_{0,\nu}\le99^{1-\nu}\|f'''\|_\infty^\nu$.
For a prescribed $M_\nu>0$, choose
$R\ge(M_{0,\nu}/M_\nu)^{1/\nu}$ and set
$f_R(x):=R^2f(x/R)$.
Then $1\le f_R''\le100$ and $[f_R'']_\nu\le M_\nu$.
The cycle becomes $-R\to0\to R\to-R$, and the zero-velocity
initialization interval becomes $[37R/50,19R/25]$.

\section{Numerical Protocol}
\label{app:numerical_protocol}

This section specifies the objectives, initialization grids,
stopping rules, and gradient accounting used in
Section~\ref{sec:numerics}.
All optimization simulations use double precision on a CPU.
The recorded environment is Python~3.12.14,
NumPy~2.3.5, SciPy~1.17.0, and Matplotlib~3.10.8.

\subsection{Cycle objectives and energy}

Let $S(z):=6z^5-15z^4+10z^3$ for $z\in[0,1]$ and
$H_0(u):=-1+(3/2)S(\min\{|u|,1\})$.
Since $S'$ and $S''$ vanish at both endpoints,
$H_0\in C^2(\mathbb R)$.
Direct differentiation gives
$\|H_0'\|_\infty=45/16$ and
$\|H_0''\|_\infty=5\sqrt3$.
Thus, with
$K_\nu:=(45/8)^{1-\nu}(5\sqrt3)^\nu$,
\[
    |H_0'(u)-H_0'(w)|
    \le
    \min\{45/8,5\sqrt3|u-w|\}
    \le
    K_\nu|u-w|^\nu.
\]
For $R_c:=(dK_\nu/M_\nu)^{1/\nu}$,
the objective in \eqref{eq:numerical_objective}
therefore satisfies $[f'']_\nu\le M_\nu$.
Its curvature lies in
$[3-(69/16)d,\,3+(21/16)d]\subset[1,7/2]$
for $0<d\le0.1$.
The numerical objective is $C^3$
and satisfies the assumptions of the main theorems.

For the three-cycle experiments,
$\mu=\gamma=M_\nu=1$ and $L=7/2$.
Write $c:=3-3d/2$, $u:=x/R_c$, and $p:=v/R_c$.
The normalized update is
\[
    p^+=(1-d)p-cu-dH_0(u),
    ~~
    u^+=u+p^+.
\]
It is independent of $\nu$ and $M_\nu$.
We compute $u_\star$ as the root of
$cu+dH_0(u)=0$ on $[0,1]$
using Brent's method with absolute tolerance $10^{-15}$.

An analytic antiderivative avoids numerical quadrature.
With $a:=|u|$, $z:=\min\{a,1\}$,
and $I(u):=\int_0^u H_0(s)\,\mathrm ds$,
\[
    I(u)
    =
    \operatorname{sgn}(u)
    \left[
        -z+\frac32
            \left(z^6-3z^5+\frac52z^4\right)
        +\frac12(a-1)_+
    \right].
\]
In the normalized scans and startup experiments,
the energy is evaluated as
\[
\begin{aligned}
    \frac{\mathcal H(R_cu,R_cp)}{R_c^2}
    ={}
    \frac c2(u-u_\star)^2
    +d\bigl[
        I(u)-I(u_\star)-H_0(u_\star)(u-u_\star)
    \bigr]
    +
    \frac A4p^2
    -\frac{1-t}{2}p[cu+dH_0(u)].
\end{aligned}
\]
The root identity makes this expression equal to
\eqref{eq:corrected_energy}.
The energy envelopes in Figures~\ref{fig:dynamics_scales}
and~\ref{fig:startup} use the sharper contraction parameter
$\widehat\lambda$ in \eqref{app:eq:lambda_sharp}.
This parameter is at least the simplified $\lambda$
stated in Theorem~\ref{thm:invariant}.

\subsection{Classification and phase scans}

The normalized equilibrium is $(u_\star,0)$,
and the three cycle phases are
$(-1,-2)$, $(0,1)$, and $(1,1)$.
The equilibrium distance is
$\sqrt{(u-u_\star)^2+p^2}$,
and the cycle distance is the minimum Euclidean distance
to these three phase states.
A trajectory receives an equilibrium or cycle label after
six consecutive post-update states have the respective distance
below $10^{-7}$.
Classified trajectories stop in the radial scans and startup comparisons.
Trajectories that do not meet either criterion by the horizon
retain an unresolved label.

The rectangular phase plot uses $\nu=1$ and $d=0.05$,
with 101 equally spaced values of $u_0$ in $[-2,2]$
and 101 equally spaced values of $p_0$ in $[-3,3]$.
All trajectories run for 1800 updates,
and classification uses their final six states.
The resulting counts are 2643 equilibrium trajectories,
7558 cycle trajectories, and zero unresolved trajectories.

For the radial scans, let
$\alpha_j:=2\pi j/64$, $j=0,\ldots,63$,
and take 96 geometrically spaced radii $\rho_\ell$
from $0.01$ to $3$, including both endpoints.
Initial states are
$(u_0,p_0):=
(u_\star+\rho_\ell\cos\alpha_j,\rho_\ell\sin\alpha_j)$.
Each trajectory receives at most $\lceil60/d\rceil$ updates.
The scan classifies every angular and radial grid point.
Table~\ref{app:tab:phase_scan} gives the complete counts.

\begin{table}[!ht]
    \centering
    \begin{tabular}{rrrrr}
        \toprule
        $d$ & Equilibrium & Cycle & Unresolved
        & Smallest cycling radius$/R_c$\\
        \midrule
        $0.100$ & 4542 & 1602 & 0 & 0.3063887\\
        $0.050$ & 4536 & 1608 & 0 & 0.2885345\\
        $0.020$ & 4538 & 1606 & 0 & 0.2885345\\
        $0.010$ & 4536 & 1608 & 0 & 0.2885345\\
        $0.005$ & 4534 & 1610 & 0 & 0.2885345\\
        $0.002$ & 4534 & 1610 & 0 & 0.2885345\\
        $0.001$ & 4534 & 1610 & 0 & 0.2885345\\
        \bottomrule
    \end{tabular}
    \caption{
    Classification of all 6144 initial states at each damping value.
    Radii are measured from the equilibrium
    in the position--velocity norm.
    }
    \label{app:tab:phase_scan}
\end{table}

At $d=0.01$, a refinement uses 128 angles
and 192 geometric radii over the same ranges.
It gives 18,131 equilibrium and 6445 cycle classifications,
with zero unresolved states.
The smallest observed cycling radius is $0.2920936R_c$,
a $1.23\%$ difference from the original scan.
The geometric radial grids are not nested,
so refinement need not decrease the smallest sampled failure distance.
The change quantifies sensitivity to this grid refinement.

For each $d$, the same normalized scan is mapped to physical
coordinates for $\nu\in\{1,3/4,1/2\}$
using the corresponding $R_c$.
The nearly constant ratio of sampled cycling distance to $R_c$ is
shown in the inset of Figure~\ref{fig:dynamics_scales}(c).
The different H\"older exponents change the physical spatial scaling.

\subsection{Startup policies and gradient accounting}

The three-cycle startup experiments use the 101 equally spaced positions
$u_0\in[-2,2]$, including endpoints, with $p_0=0$.
GD uses step size $1/L=2/7$,
direct HB uses $\gamma=1$ and $\beta=1-d$,
and the hybrid starts with GD.
Before each hybrid update, the current gradient is checked against
\eqref{eq:momentum_start}.
When the test first passes, the velocity is reset to zero
and that same gradient is used for the first HB update.
Each optimization update therefore costs one gradient evaluation.
Offline energy diagnostics are excluded from the counts for every method.
For GD, the velocity in the common state-error criterion
is the latest displacement.

All methods use the same six-state confirmation rule
and horizon $\lceil60/d\rceil$.
Table~\ref{app:tab:startup_all} reports all nine parameter combinations.
GD and the hybrid reach the equilibrium tolerance from every start,
giving 909 successful hybrid runs.
The direct-HB classifications and GD counts are reused under
the physical rescalings because their normalized updates
and stopping rules are independent of $\nu$.
The hybrid startup threshold depends on $\nu$.

\begin{table}[!ht]
    \centering
    \begin{tabular}{ccccrr}
        \toprule
        $\beta$ & $\nu$ & HB cycles & GD warm-up
        & Median GD calls & Median hybrid calls\\
        \midrule
        $0.90$ & $1$   & $83/101$ & $1$--$5$ & 16 & 213\\
        $0.90$ & $3/4$ & $83/101$ & $2$--$5$ & 16 & 189\\
        $0.90$ & $1/2$ & $83/101$ & $3$--$7$ & 16 & 149\\
        $0.95$ & $1$   & $83/101$ & $1$--$4$ & 15 & 426\\
        $0.95$ & $3/4$ & $83/101$ & $2$--$5$ & 15 & 402\\
        $0.95$ & $1/2$ & $83/101$ & $3$--$6$ & 15 & 299\\
        $0.99$ & $1$   & $84/101$ & $0$--$4$ & 15 & 2239\\
        $0.99$ & $3/4$ & $84/101$ & $1$--$4$ & 15 & 2012\\
        $0.99$ & $1/2$ & $84/101$ & $2$--$5$ & 15 & 1620\\
        \bottomrule
    \end{tabular}
    \caption{
    All startup cases.
    GD warm-up is the range of updates before resetting
    the velocity and starting HB.
    Median costs include the six-state stopping confirmation.
    }
    \label{app:tab:startup_all}
\end{table}

Across the 909 hybrid runs,
the largest switch ratio
$\|\nabla f(x)\|^2/(2\mu\mathcal B_\nu)$
is approximately $0.99793$.
The largest reset-state energy ratio
$\mathcal H(x,0)/\mathcal B_\nu$ is $0.34950$,
and the largest energy ratio after an HB update is $0.32867$.
The maximum post-switch speed is $0.34679R_\nu$.
Recorded post-switch energy increments are nonpositive
to numerical precision.

For the representative initialization
$(x_0,v_0)=(R_c,0)$ at $d=0.05$ and $\nu=1$,
GD reaches the stopping criterion after 15 calls.
Direct HB is classified as cycling after 658 calls.
The hybrid switches after three GD calls
and reaches the equilibrium criterion after 467 total calls.
In the post-switch energy plot,
the initial energy is recomputed at the reset state $(x,0)$;
the incoming GD displacement is not used as HB's initial velocity.

\subsection{Invariant-region checks}

With seed 104729, we sample 5000 states at each
$\beta\in\{0.9,0.95,0.99,0.999\}$ to check the energy identities.
The largest corrected-identity residual divided by $R_c^2$ is
$5.15\times10^{-14}$, and the largest cycle residual divided by
$R_c$ is $5.13\times10^{-16}$.
For each parameter, 512 uniformly sampled states in the
$d/(40M)$ disk are followed for 240 updates.
The energy budget, energy descent, and displacement bounds hold to
numerical precision throughout these trajectories.
At $\beta=0.999$, the largest energy ratio at the end of this
finite horizon is $0.787$.
The separate exact-arithmetic check of the zero-momentum construction
in Appendix~\ref{app:zero_init} returns PASS.

\subsection{Optimal-parameter startup}
\label{app:polyak_numerics}

The final experiment uses the family in
\eqref{eq:numerical_local_objective}.
Since $S'(s)=30s^2(1-s)^2$ on $[0,1]$,
$\max S'=15/8$. The choice $a_\kappa=15(\kappa-1)/8$
therefore gives $1\le f_\kappa''\le\kappa$ and
$\operatorname{Lip}(f_\kappa'')=1$.
Both endpoint curvatures are attained.
The clipped smoothstep is $C^2$, so $f_\kappa\in C^4$.
The minimizer is zero, and $f_\kappa(x)=x^2/2$ for
$|x|\le a_\kappa$.

For stable evaluation, write $y:=x/a_\kappa$, $u:=v/a_\kappa$,
$z:=|y|$, and $r:=\operatorname{clip}(z-1,0,1)$.
Define $P(r):=(5/2)r^4-3r^5+r^6$.
Integration of the prescribed Hessian gives
\[
    \frac{f_\kappa'(a_\kappa y)}{a_\kappa}
    =y+(\kappa-1)\operatorname{sgn}(y)
       \bigl[P(r)+(z-2)_+\bigr].
\]
The normalized objective gap is
\[
    \frac{f_\kappa(a_\kappa y)}{a_\kappa^2}
    =\frac{y^2}{2}+(\kappa-1)J(z),
\]
where
\[
    J(z):=
    \begin{cases}
        0,&0\le z\le1,\\
        \frac12(z-1)^5-\frac12(z-1)^6+\frac17(z-1)^7,
            &1<z\le2,\\
        \frac17+\frac12(z-2)+\frac12(z-2)^2,&z>2.
    \end{cases}
\]
These formulas avoid numerical quadrature and subtraction of nearby
objective values.

We use $\kappa\in\{4,16,64,256\}$ and initial positions
$x_0/a_\kappa\in\{1.5,2,3\}$, with $v_0=0$.
The main-text comparison uses $x_0=2a_\kappa$.
GD baselines use either $1/L$ or $2/(L+\mu)$.
The hybrid always uses $1/L$ for warm-up, then the quadratic-optimal
HB parameters for $[\mu,L]=[1,\kappa]$.
The startup budget uses the general-$\gamma$ definitions in
\eqref{eq:energy_budget}, with $\nu=M=1$.
The current gradient is reused for the startup test and the update.
We require six consecutive post-update states with
$\sqrt{x_k^2+v_k^2}<10^{-10}a_\kappa$ and allow at most
$10^5$ updates. Every run meets the stopping criterion.
GD's state velocity is its latest displacement.
At the switching time, traces record the reset velocity zero.

The four main runs switch at normalized positions
$0.1500996$, $0.02037265$, $0.002497164$, and $0.0002985594$,
respectively. These positions lie in the quadratic core.
If $T$ is the switching time and
$q:=(\sqrt\kappa-1)/(\sqrt\kappa+1)$, then
\begin{equation}
    \begin{aligned}
        x_{T+j}&=x_T\bigl[1+(1-q)j\bigr]q^j,\\
        v_{T+j}&=-x_T(1-q)^2j q^{j-1},
        ~~ j\ge0.
    \end{aligned}
    \label{app:eq:numerical_tail}
\end{equation}
Indeed, the characteristic polynomial in the core is $(z-q)^2$,
and velocity resetting gives $x_{T-1}=x_T$.
The position formula is positive and nonincreasing when $x_T>0$,
so the trajectory remains in the core.
Its asymptotic root rate is $q$; consecutive state-norm ratios include
the repeated-root transient.
Write $e_j:=\|(x_{T+j},v_{T+j})\|_2$ for the post-switch state error.
For the main runs, the largest relative discrepancy between the
computed and exact state in \eqref{app:eq:numerical_tail} is
$3.05\times10^{-14}$; across all tested starts, it is below
$3.8\times10^{-14}$.
The energy budget, next-displacement bound, and energy contraction
from Theorem~\ref{thm:invariant} hold throughout the HB phases.

\begin{figure}[ht]
    \centering
    \includegraphics[width=\textwidth]{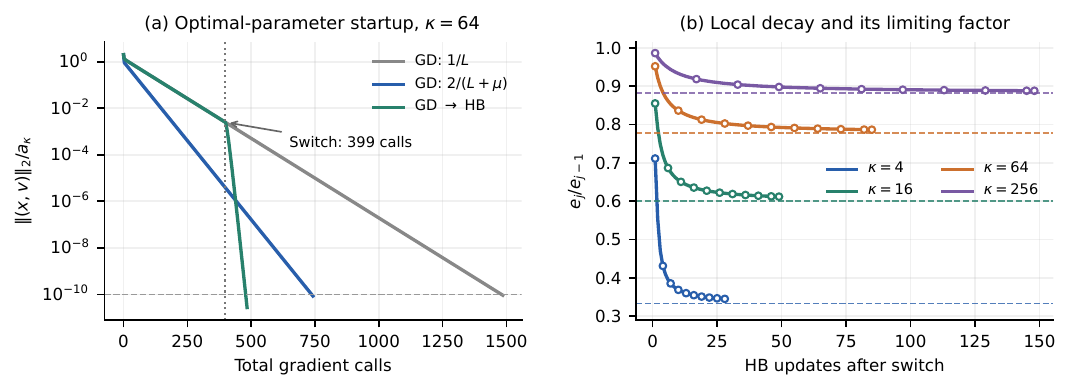}
    \caption{
    Optimal-parameter startup.
    (a) Total gradient cost on $f_{64}$ from $(2a_{64},0)$;
    the hybrid switches after 399 GD updates.
    (b) Consecutive post-switch state-error ratios for the four
    condition numbers. Symbols show numerical trajectories,
    solid curves show the exact repeated-root solution, and
    dashed lines mark the limiting factors $q$.
    }
    \label{app:fig:polyak_startup}
\end{figure}

Table~\ref{app:tab:polyak_all} gives all tested initializations.
The total-cost comparison depends on the initial position:
for $x_0=3a_\kappa$, GD with step $2/(L+\mu)$ reaches
$x_1=0$ exactly because
$f_\kappa'(3a_\kappa)=3(\kappa+1)a_\kappa/2$.
It then meets the state-error and six-state confirmation rule
in seven calls. For warm-up with step $1/L$, positions
$2a_\kappa$ and $3a_\kappa$ produce the same first iterate,
which explains their identical subsequent hybrid trajectories.

\begin{table}[ht]
    \centering
    \begin{tabular}{ccccccc}
        \toprule
        $\kappa$ & $x_0/a_\kappa$ & GD $1/L$ & GD $2/(L+\mu)$
        & Warm-up & HB phase & Hybrid total\\
        \midrule
        $4$ & $1.5$ & 87 & 52 & 8 & 28 & 36\\
        $4$ & $2$ & 87 & 51 & 8 & 28 & 36\\
        $4$ & $3$ & 87 & 7 & 8 & 28 & 36\\
        $16$ & $1.5$ & 367 & 192 & 65 & 49 & 114\\
        $16$ & $2$ & 368 & 190 & 66 & 49 & 115\\
        $16$ & $3$ & 368 & 7 & 66 & 49 & 115\\
        $64$ & $1.5$ & 1485 & 751 & 398 & 85 & 483\\
        $64$ & $2$ & 1486 & 742 & 399 & 85 & 484\\
        $64$ & $3$ & 1486 & 7 & 399 & 85 & 484\\
        $256$ & $1.5$ & 5943 & 2979 & 2128 & 148 & 2276\\
        $256$ & $2$ & 5944 & 2953 & 2129 & 148 & 2277\\
        $256$ & $3$ & 5944 & 7 & 2129 & 148 & 2277\\
        \bottomrule
    \end{tabular}
    \caption{
    All optimal-parameter startup runs. Entries count gradient calls;
    total costs include stopping confirmation.
    The hybrid's warm-up uses $1/L$ in every case.
    }
    \label{app:tab:polyak_all}
\end{table}

\end{document}